\documentclass[11pt]{article}
\usepackage[english]{babel}
\usepackage[latin1]{inputenc}
\usepackage{amsmath}
\usepackage{amsfonts}
\usepackage{amssymb}
\usepackage{amsthm}
\usepackage{mathrsfs}
\usepackage{enumitem}
\usepackage{graphicx}

\newtheorem{thm}{Theorem}
\newtheorem{defn}{Definition}

\newtheorem{prop}{Proposition}
\newtheorem{lem}{Lemma}
\newtheorem{rem}{Remark}
\usepackage{hyperref}
\usepackage{color}    
\hypersetup{
colorlinks=true,       
    linkcolor=red,          
    citecolor=darkgreen,        
    filecolor=BrickRed,      
    urlcolor=blue        
}
\definecolor{darkgreen}{rgb}{0,0.4,0}

\renewcommand{\geq}{\geqslant}
\renewcommand{\leq}{\leqslant}

\begin{document}

\title{\LARGE{{\bf Counting walks in a quadrant: a unified
               approach via boundary value problems}}}

\author{Kilian Raschel\footnote{Laboratoire de Probabilit\'es et
        Mod\`eles Al\'eatoires, Universit\'e Pierre et Marie Curie,
        4 Place Jussieu, 75252 Paris Cedex 05, France.
        E-mail: \href{mailto:kilianraschel@yahoo.fr}{kilian.raschel@upmc.fr}}}

\date{February 13, 2011}

\maketitle


{\small
\noindent{{\bf Abstract.}
     The aim of this article is to introduce a unified method to obtain
     explicit integral representations of the trivariate generating function
     counting the walks with small steps which are confined to a quarter plane.
     For many models, this yields for the first time an explicit
     expression of the counting generating function. Moreover, the nature of the integrand
     of the integral formulations is shown to be directly
     dependent on the finiteness of a naturally attached group
     of birational transformations as well as on the sign of the covariance of the walk. }}

\medskip


\vspace{1.5mm}

\section{Introduction}
\label{Introduction}

The enumeration of planar lattice walks is a classical topic in combinatorics. For
a given set $\mathcal{S}$ of steps, it is a matter of counting the numbers of paths
of a certain length, having jumps in $\mathcal{S}$,  starting
and ending at some arbitrary points, and possibly restricted to a region of the plane.
There are two main questions:
\begin{itemize}
\item[---] How many such paths exist?
\item[---] Is the underlying generating function rational, algebraic,
holonomic (\textit{i.e.}, a solution of a linear differential equation with polynomial
coefficients) or non-holonomic?
\end{itemize}
For instance, if the paths are not restricted to a region,
or if they are constrained to a half-plane, the counting generating function
can then be made explicit and is rational or algebraic \cite{MBM2},
respectively.
It is natural to consider then the walks confined to the intersection of two half-planes,
as the quadrant $\mathbb{Z}_{+}^{2}$. The situation seems richer: some
walks admit an algebraic counting function, see \cite{FL1,Gessel}
for the walk with unit step set $\mathcal{S}=\{{\sf W}, {\sf NE}, {\sf S}\}$,
while others admit a counting function that is not even holonomic, see \cite{MBM2}
for the knight walk. Henceforth, we focus on these walks staying in $\mathbb{Z}_{+}^{2}$.

Bousquet-M\'{e}lou and Mishna have recently \cite{BMM}
initiated a systematic study of the walks confined to $\mathbb{Z}_{+}^{2}$, starting at
the origin and having small steps, which means that the set of admissible steps $\mathcal{S}$
is included in the set of the eight nearest neighbors, \textit{i.e.},
     \begin{equation*}
          \mathcal{S}\subset \{{\sf W}, {\sf NW}, {\sf N},{\sf NE}, {\sf E}, {\sf SE},{\sf S}, {\sf SW}\}\footnote{For the step set $\mathcal{S}$, we shall equally denote it
by using the cardinal points $\{{\sf W}, {\sf NW}, {\sf N},{\sf NE}, {\sf E}, {\sf SE},{\sf S}, {\sf SW}\}$ or
the cartesian points $\{-1,0,1\}^{2}\setminus \{(0,0)\}$.}.
     \end{equation*}
On the boundary, the jumps are the natural
ones: the steps that would take the walk out $\mathbb{Z}_{+}^{2}$ are discarded.
\begin{figure}[t]
\begin{center}
\begin{picture}(340.00,68.00)
\includegraphics{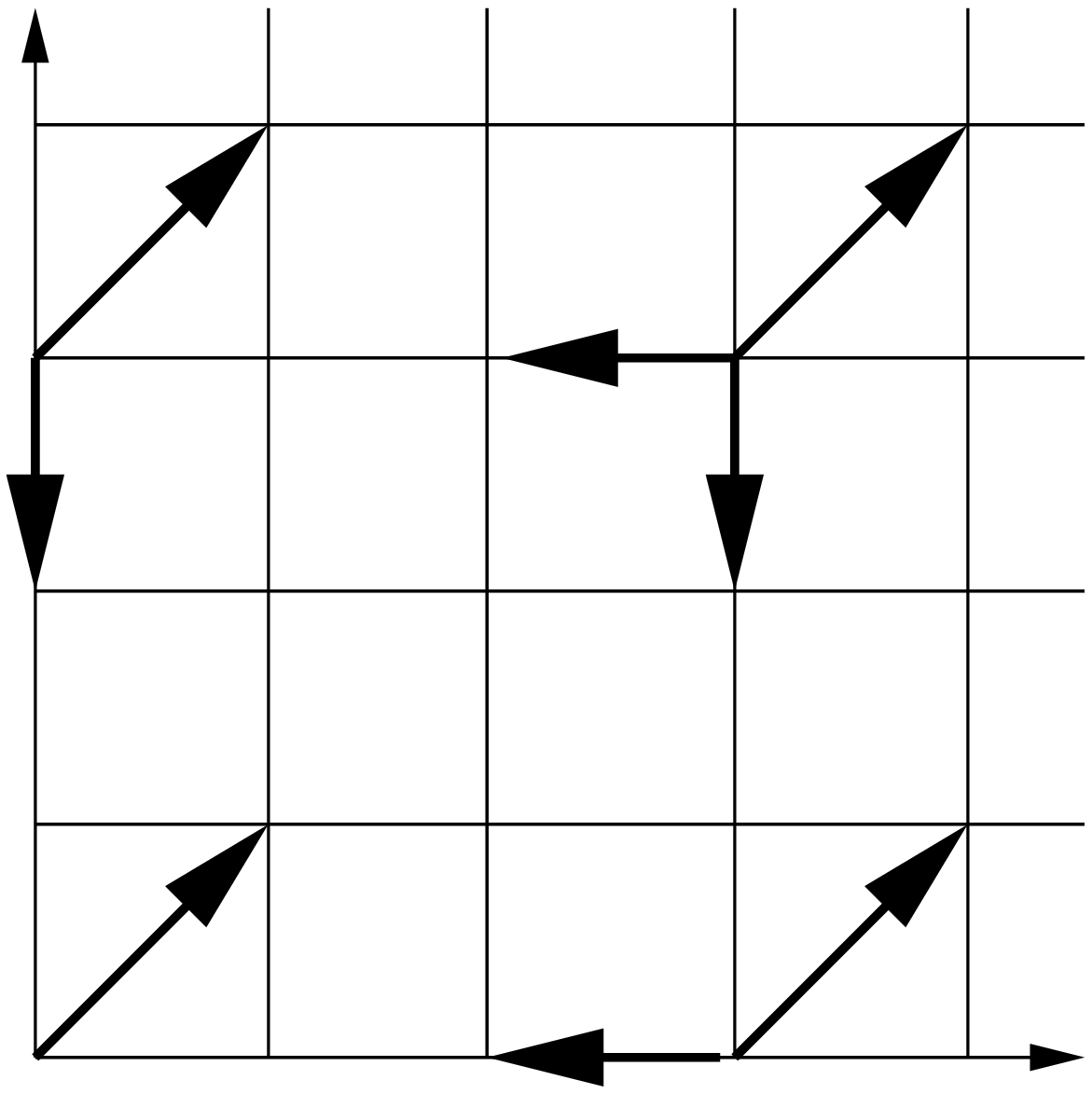}
\hspace{45mm}
\includegraphics{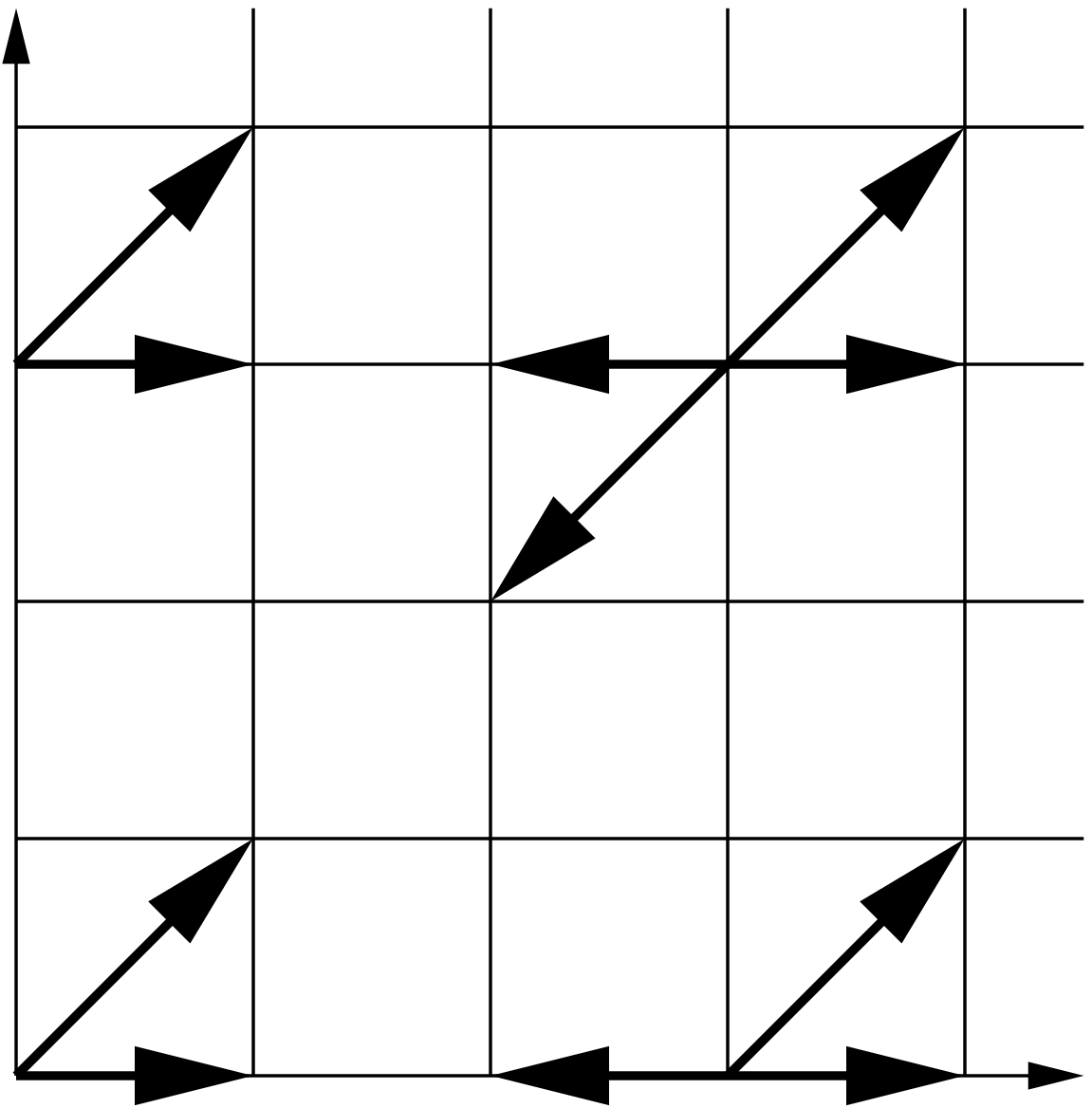}
\hspace{45mm}
\includegraphics{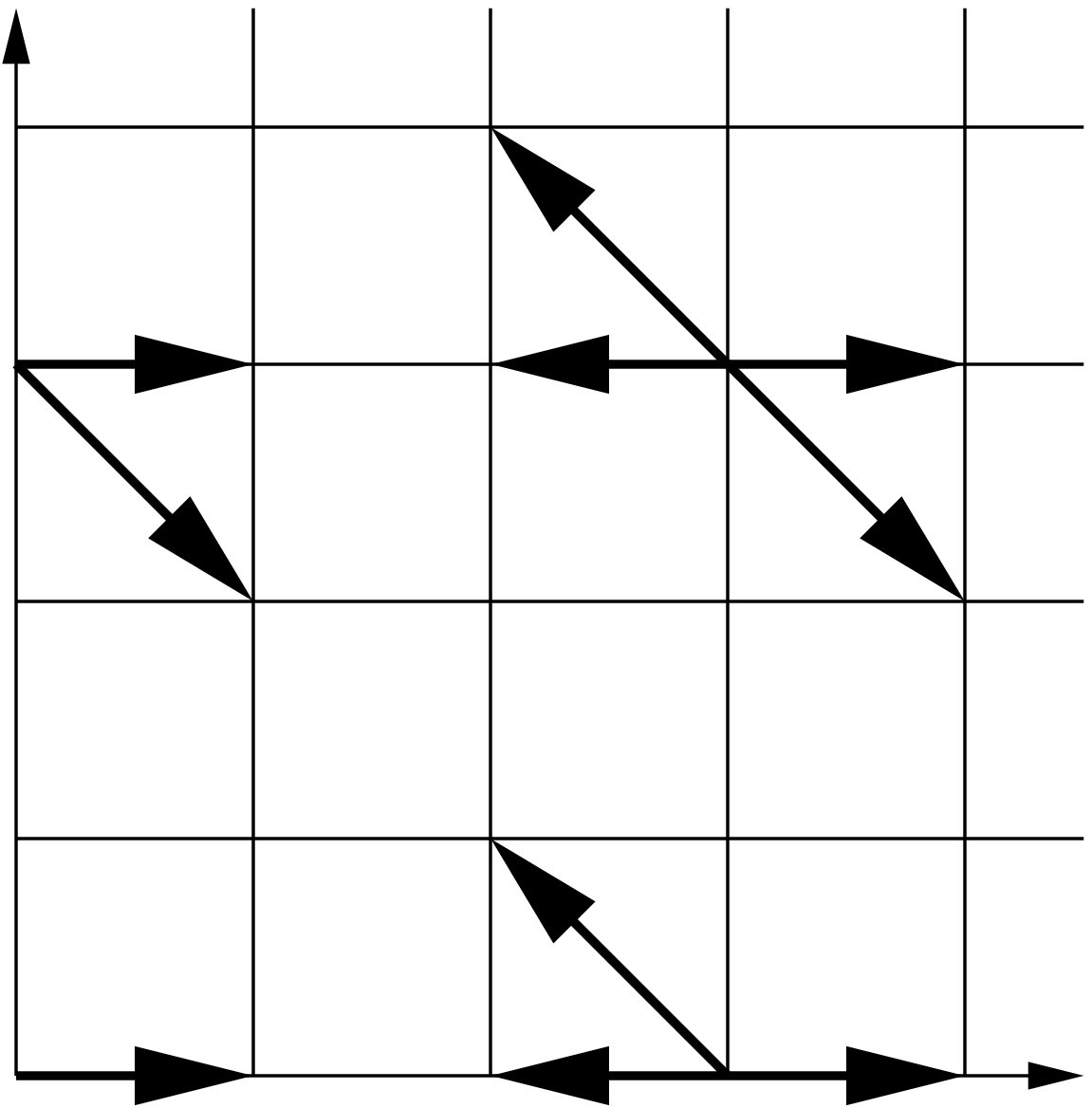}
\end{picture}
\end{center}
\caption{Three famous examples, known as Kreweras', Gessel's and Gouyou-Beauchamps' walks.
They have been, and
are still, the object of many studies, see, \textit{e.g.}, \cite{Mi,FL2,FL1}, \cite{BK,KKZ,KRG} and
\cite{Stache,KRSp4}, respectively.}
\label{ExExEx}
\end{figure}
There are $2^{8}$ such models.
Of these, there are obvious symmetries and Bousquet-M\'{e}lou and Mishna \cite{BMM}
show that there are in fact $79$ types of essentially distinct walks---we will often refer to these $79$ walks tabulated in \cite{BMM}.

The central object for the study of these $79$ walks is the following: denoting
by $q(i,j;n)$ the number of paths confined to $\mathbb{Z}_{+}^{2}$, having length $n$, starting at $(0,0)$ and ending at
$(i,j)$, their \emph{generating function} is defined as
     \begin{equation}
     \label{definition_generating_function}
          Q(x,y;z)=\sum_{i,j,n\geq 0}q(i,j;n)x^{i}y^{j}z^{n}.
     \end{equation}
\begin{prop}
We have
     \begin{equation}
     \label{functional_equation}
          xyz\left[\sum_{(i,j)\in\mathcal{S}}x^{i}y^{j}-1/z\right]Q(x,y;z)=
          c(x;z)Q(x,0;z)+
          \widetilde{c}(y;z)Q(0,y;z)-
          z\delta Q(0,0;z)-x y,
     \end{equation}
where we have noted
     \begin{equation}
     \label{c_c_tilde}
          c(x;z)=zx\sum_{(i,-1)\in\mathcal{S}}x^{i},
          \ \ \ \ \ \widetilde{c}(y;z)=zy\sum_{(-1,j)\in\mathcal{S}}y^{j},
          \ \ \ \ \ \delta =\left\{\begin{array}{ccc}
          1 & \text{if} &         {\sf SW}\in\mathcal{S},\\
          0 & \text{if} &   {\sf SW}\notin\mathcal{S}.
          \end{array}\right.
     \end{equation}
\end{prop}
This functional equation is the fundamental starting point of our study---and is also
so for almost every other works on the topic. It relates the trivariate generating function
$Q(x,y;z)$ to the bi- and univariate generating functions $Q(x,0;z)$,
$Q(0,y;z)$ and $Q(0,0;z)$ counting the walks which end on the borders.
Notice that \eqref{functional_equation} simply follows from the step by
step construction of the walks; its proof may be found in \cite[Section 4]{BMM}.
The polynomial
     \begin{equation}
     \label{def_kernel}
          xyz\left[\sum_{(i,j)\in\mathcal{S}}x^{i}y^{j}-1/z\right]
     \end{equation}
appearing in \eqref{functional_equation} is usually called the \emph{kernel} of the walk.
If $k$ is the cardinal of $\mathcal{S}$, then \eqref{functional_equation}
is valid at least on $\{|x|\leq 1,|y|\leq 1, |z|<1/k\}$, since clearly $q(i,j;n)\leq {k}^{n}$.

In this way, for the purpose of answering both questions stated at the
beginning of this paper, it suffices to
solve \eqref{functional_equation}. A key idea then is to consider a
certain group, introduced in \cite{MAL} in a probabilistic context,
and called the \emph{group of the walk}. This group of
birational transformations of $\mathbb{C}^{2}$,  which leaves
invariant the step generating function $\sum_{(i,j)
\in\mathcal{S}}x^{i}y^{j}$, is the group $W=\langle
\Psi,\Phi\rangle$ generated by 
     \begin{equation}
     \label{def_generators_group}
          \Psi(x,y)=\left(x,\frac{\sum_{(i,-1)\in\mathcal{S}}x^{i}}
          {\sum_{(i,+1)\in\mathcal{S}}x^{i}}\frac{1}{y}\right),
          \ \ \ \ \ \Phi(x,y)=\left(\frac{\sum_{(-1,j)\in\mathcal{S}}y^{j}}
          {\sum_{(+1,j)\in\mathcal{S}}y^{j}}\frac{1}{x},y\right).
     \end{equation}
Obviously $\Psi\circ\Psi=\Phi\circ\Phi=\text{id}$, and $W$ is a dihedral group---of
order even and at least four. This order is calculated in \cite{BMM} for each
of the 79 cases: 23 walks admit a finite group (of order four, six or eight),
and the 56 others have an infinite group.

For the 23 walks with a finite group, the answers to both questions (regarding explicit expression and
nature of the function \eqref{definition_generating_function}) have been given recently. Indeed, the
article \cite{BMM} successfully treats 22 of the 23 models associated with a finite group: the series
\eqref{definition_generating_function} is made explicit and is shown to be either algebraic
or transcendental but holonomic. As for the 23rd walk (namely, Gessel's walk represented in
Figure \ref{ExExEx}), Bostan and Kauers \cite{BK} have given a computer-aided proof of the algebraicity of the function
\eqref{definition_generating_function}. Furthermore, using a powerful computer
algebra system, they have made explicit minimal polynomials. Thanks to these polynomials, van
Hoeij \cite[Appendix]{BK} has then managed to express the function \eqref{definition_generating_function} by radicals. At the same time, we gave in \cite{KRG}
an explicit integral representation of \eqref{definition_generating_function} for Gessel's walk,
this without computer help.
Based on ideas of \cite[Chapter 4]{FIM}, alternative proofs for the nature of \eqref{definition_generating_function}
for these 23 cases are given in \cite{FR}.
Moreover, in the work in preparation \cite{CET}, Bostan \emph{et al.}\ obtain
integral representations of the function \eqref{definition_generating_function} for the 23 walks
having a finite group, by using a mathematical and algorithmic method, based on creative telescoping
and on the resolution of differential equations of order two in terms of hypergeometric functions.

\begin{figure}[t]
\begin{center}
\begin{picture}(427.00,68.00)
\includegraphics{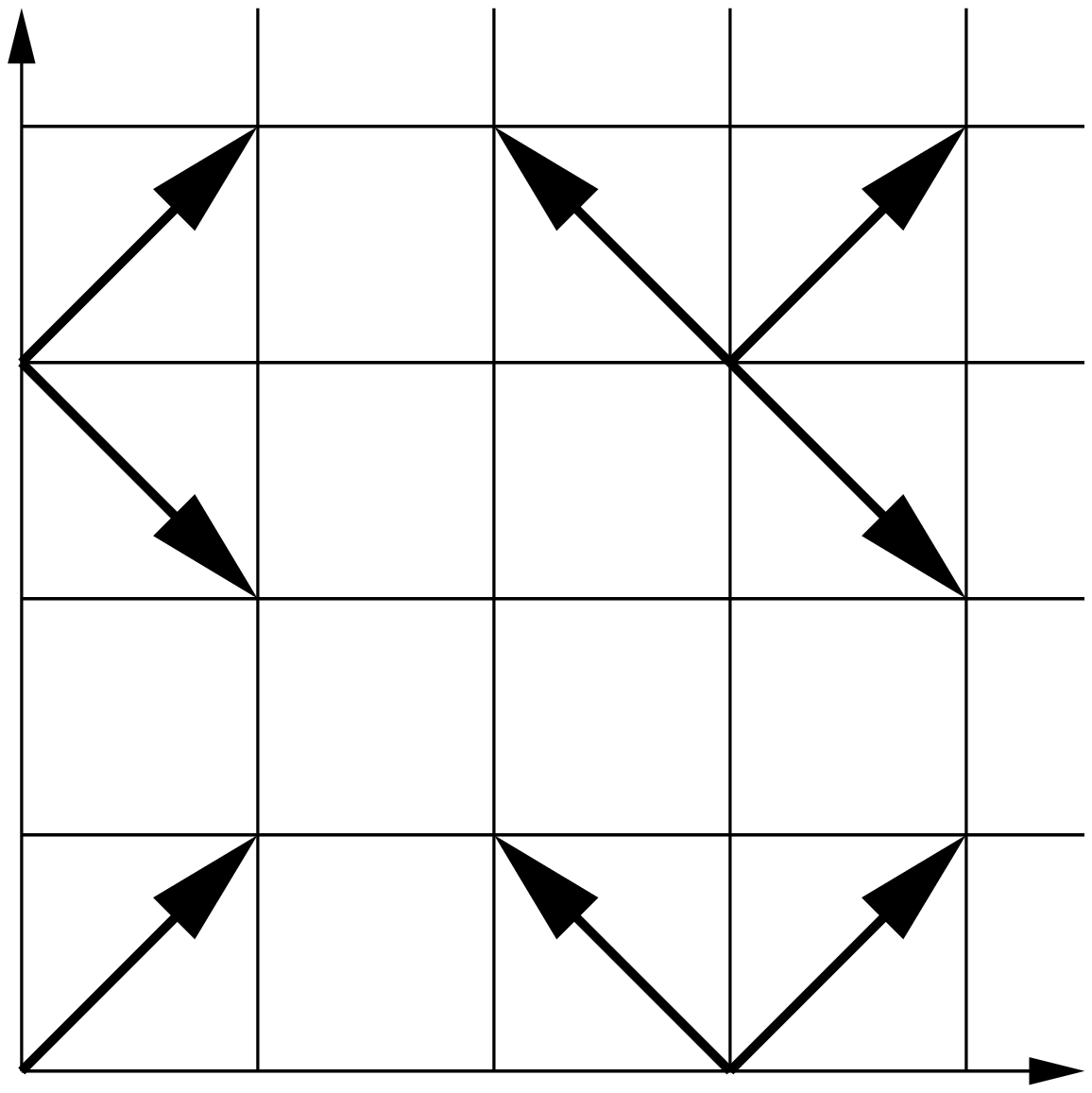}
\hspace{29mm}
\includegraphics{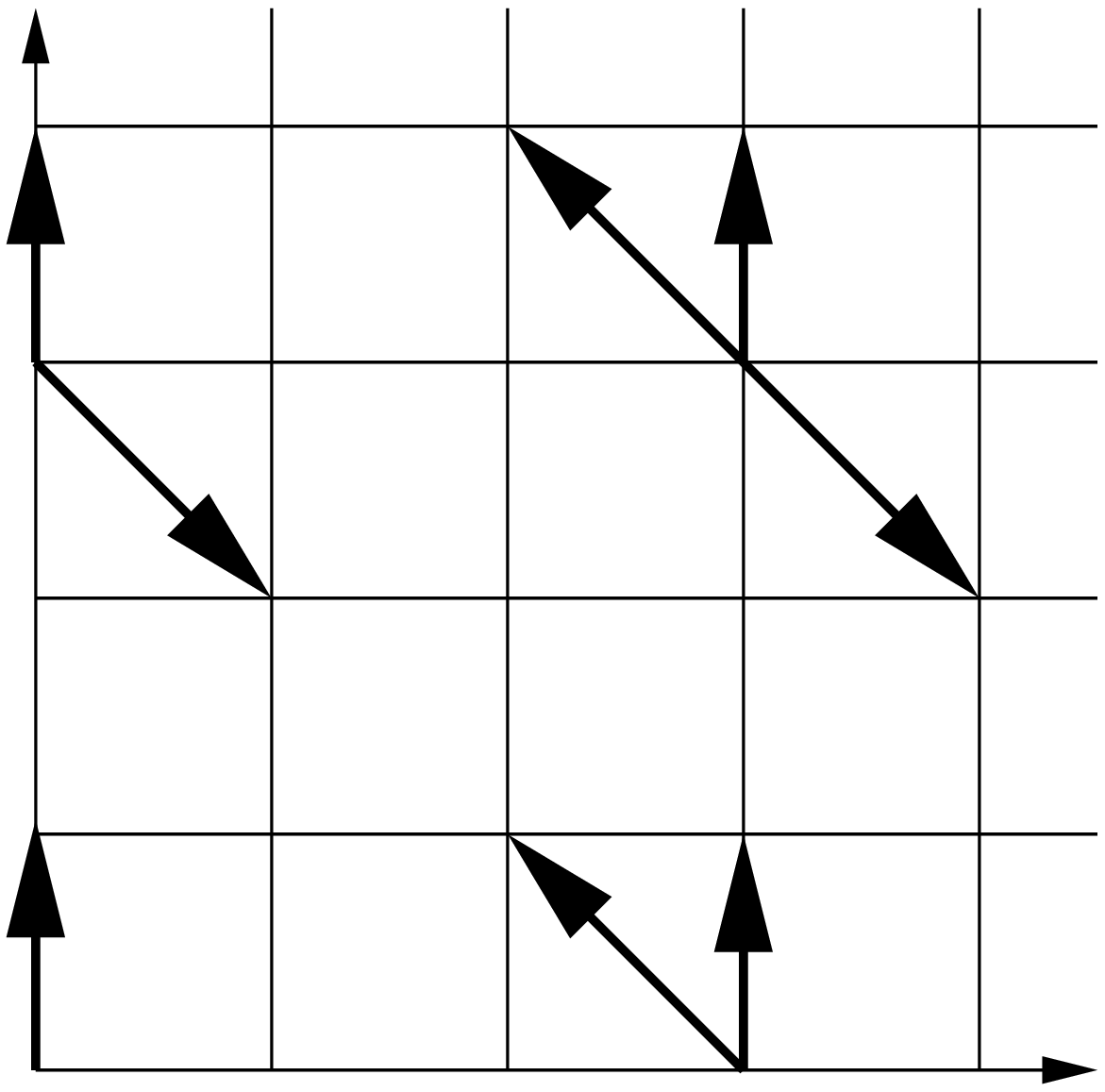}
\hspace{29mm}
\includegraphics{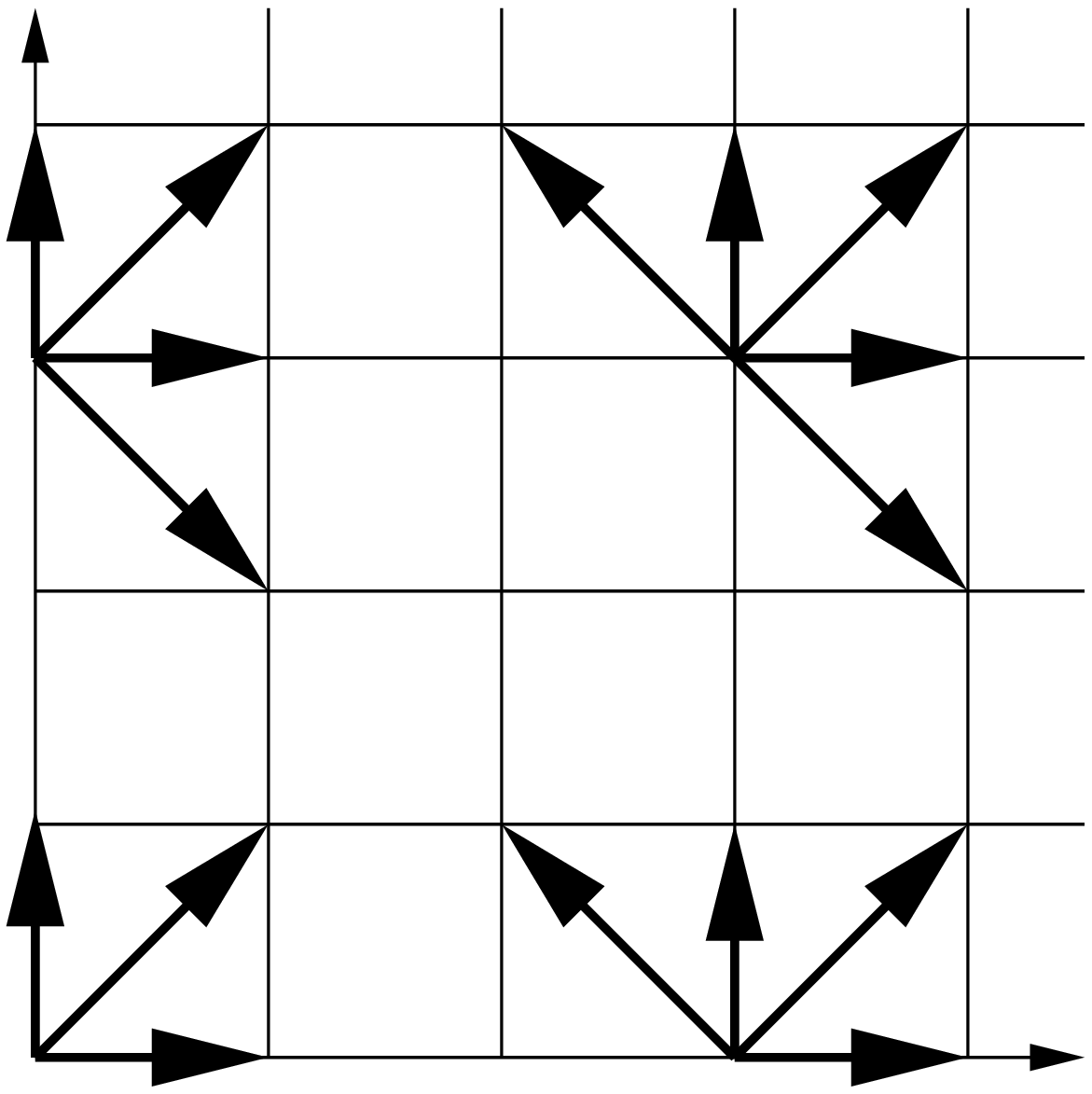}
\hspace{29mm}
\includegraphics{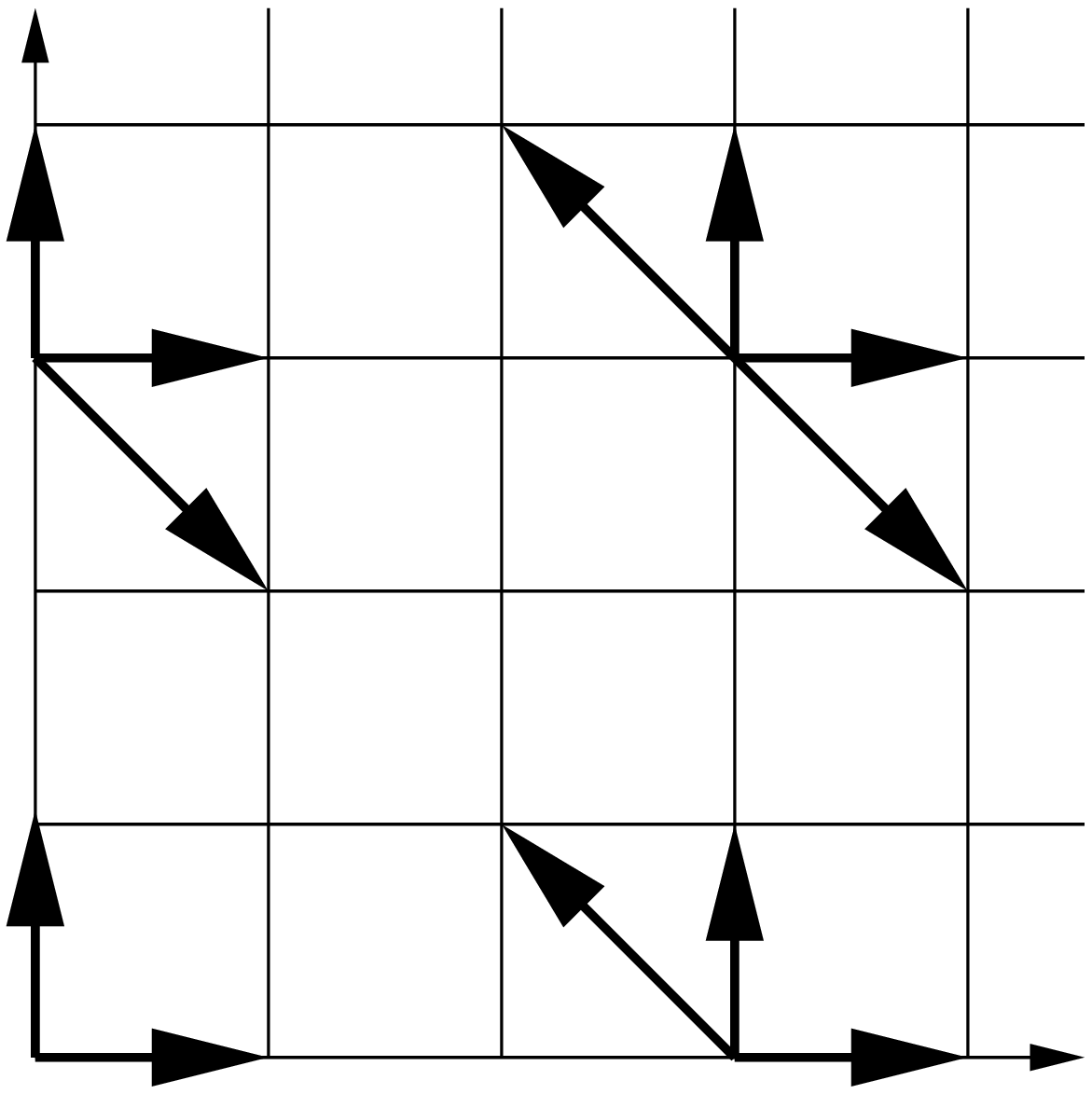}
\hspace{29mm}
\includegraphics{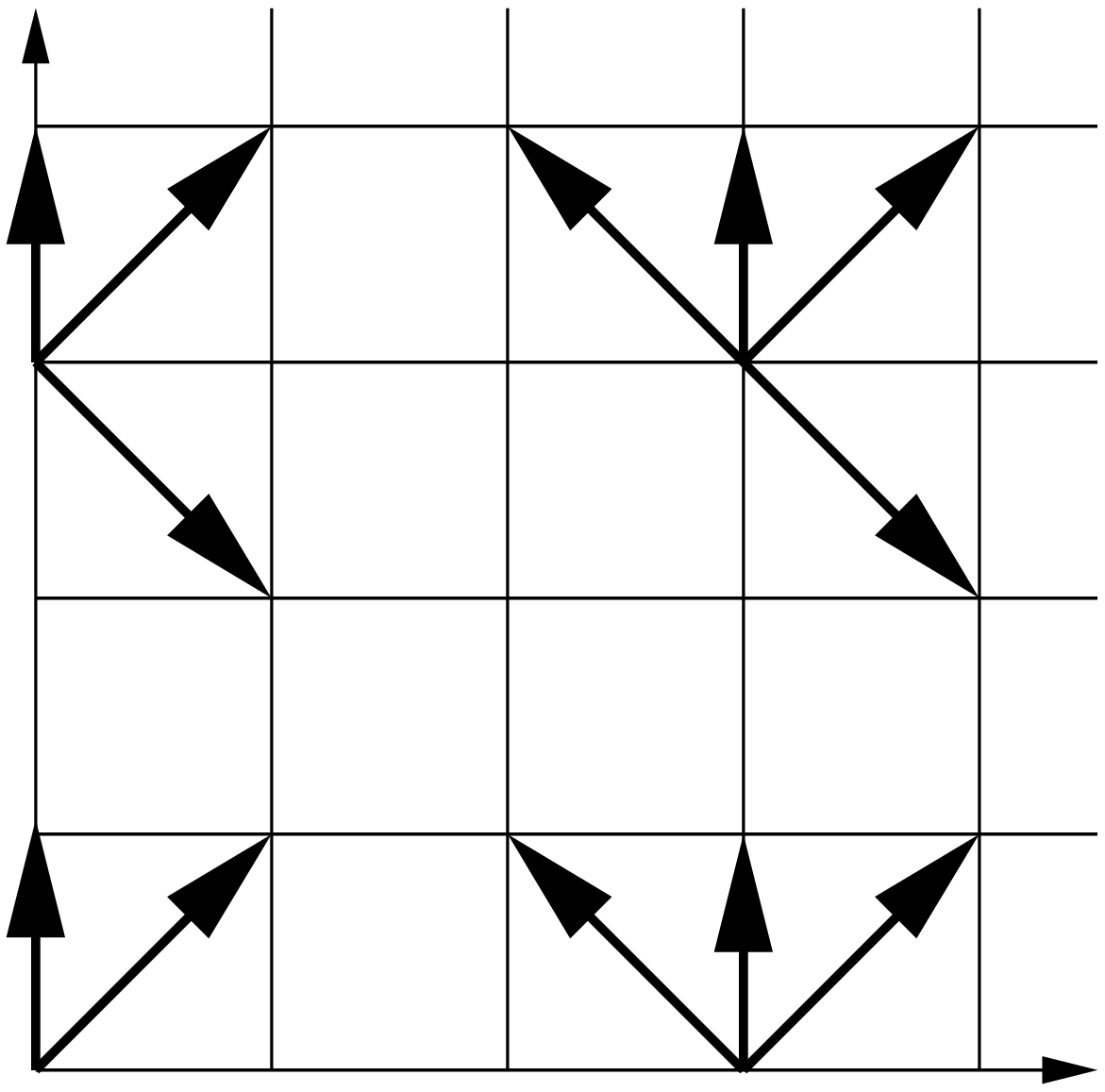}
\end{picture}
\end{center}
\caption{The $5$ singular walks in the classification of \cite{BMM}}
\label{The_five_singular_walks}
\end{figure}

Concerning the $56$ walks with an infinite group, the 5 represented in Figure \ref{The_five_singular_walks}
are special; they are called \emph{singular}.
Of these, $2$ cases are solved: in \cite{MM2}, Mishna and Rechnitzer have considered the walks with step sets $\mathcal{S}=$
$\{{\sf NW},{\sf NE},{\sf SE}\}$ and $\{{\sf NW},{\sf N},{\sf SE}\}$, have made explicit the series
\eqref{definition_generating_function}, and have shown that it is non-holonomic.
What is left is $54$ walks whose status is currently unsettled, as regards an explicit expression as well as
holonomicity. However, certain asymptotic conjectures are proposed in \cite{BKA}.

\section{Main results}

The aim of this article is to introduce a unified approach giving an explicit expression
of the generating function \eqref{definition_generating_function} for any of the 79 walks.
We start with the non-singular walks. Please note that important notations appear below the statement.

\begin{thm}
\label{explicit_integral}
Assume that the walk is one of the 74 non-singular walks.
\begin{enumerate}
\item \label{moim} The following integral representations relative to
      $Q(x,0;z)$ and $Q(0,y;z)$ hold:
     \begin{align*}
          &\hspace{-1mm}c(x;z)Q(x,0;z)-c(0;z)Q(0,0;z)=
          \\&\hspace{-1mm}xY_{0}(x;z)+\frac{1}{2\imath \pi}\int_{x_{1}(z)}^{x_{2}(z)}[Y_0(t;z)-Y_1(t;z)]\left[\frac{\partial_{t}w(t;z)}{w(t;z)-w(x;z)}
          -\frac{\partial_{t}w(t;z)}{w(t;z)-w(0;z)}\right]\textnormal{d}t,\\
          &\hspace{-1mm}\widetilde{c}(y;z)Q(0,y;z)-\widetilde{c}(0;z)Q(0,0;z)=
          \\&\hspace{-1mm}X_{0}(y;z)y+\frac{1}{2\imath \pi}\int_{y_{1}(z)}^{y_{2}(z)}
          [X_0(t;z)-X_1(t;z)]
          \left[\frac{\partial_{t}\widetilde{w}(t;z)}{\widetilde{w}(t;z)-
          \widetilde{w}(y;z)}-\frac{\partial_{t}\widetilde{w}(t;z)}
          {\widetilde{w}(t;z)-\widetilde{w}(0;z)}\right]
          \textnormal{d}t.
     \end{align*}

\item The value of $Q(0,0;z)$ is determined as follows.
\begin{itemize}
\item[---] If ${\sf SW}\in \mathcal{S}$\footnote{This condition is equivalent to $\delta=1$ in \eqref{functional_equation},
and to $c(0;z)=\widetilde{c}(0;z)=z$, see \eqref{c_c_tilde}.}, then for any $(x_0,y_0,z_0)\in\{|x|\leq 1,|y|\leq 1, |z|<1/k\}$ at which the
kernel \eqref{def_kernel} vanishes, we have
     \begin{align*}
          Q(0,0;z)=x_0y_0/z-[c(x_0;z)Q(x_0,0;z)&-c(0;z)Q(0,0;z)]/z
          \\-[\widetilde{c}(\hspace{0.3mm}y_0;z)Q(0,\hspace{0.3mm}y_0;z)&-\widetilde{c}(0;z)Q(0,0;z)]/z.
     \end{align*}
\item[---] If ${\sf SW}\notin \mathcal{S}$, then
     \begin{equation*}
          Q(0,0;z)=\lim_{x\to 0} \frac{c(x;z)Q(x,0;z)-c(0;z)Q(x,0;z)}{c(x;z)}.
     \end{equation*}
\end{itemize}
\item The function $Q(x,y;z)$ has the explicit expression
     \begin{equation*}
          Q(x,y;z)=
          \frac{c(x;z)Q(x,0;z)+
          \widetilde{c}(y;z)Q(0,y;z)-
          z\delta Q(0,0;z)-x y}{xyz[\sum_{(i,j)\in\mathcal{S}}x^{i}y^{j}-1/z]}.
     \end{equation*}
\end{enumerate}
\end{thm}
In the statement of Theorem \ref{explicit_integral}:
\begin{itemize}
\item[---] $c$ and $\widetilde{c}$ are defined in \eqref{c_c_tilde}.
\item[---] $Y_{0}$ and $Y_{1}$ (resp.\ $X_{0}$ and $X_{1}$)
           are the $y$- (resp.\ $x$-) roots of the kernel \eqref{def_kernel},
           which is a second-degree polynomial. They are chosen such that $|Y_0|\leq |Y_1|$
           (resp.\ $|X_0|\leq |X_1|$), see
           Lemma \ref{Properties_X_Y_0}. Their expression is given
           in \eqref{def_XX} (resp.\ \eqref{def_YY}).
\item[---] $x_1(z)$ and $x_2(z)$ (resp.\ $y_1(z)$ and $y_2(z)$) are two of the four 
           \emph{branch points}
           of $Y_0$ and $Y_1$ (resp.\ $X_0$ and $X_1$), see Subsection \ref{k} below \eqref{d_d_tilde} for their
           proper definition. These four branch points can be
           characterized as the only points satisfying $Y_0=Y_1$ (resp.\ $X_0=X_1$).
           An equivalent definition is that they are the roots of the discriminant of the kernel \eqref{d_d_tilde}
           viewed as a second-degree polynomial in the variable $y$ (resp.\ $x$), see \eqref{kernel_pt}.
\item[---] $w$ and $\widetilde{w}$ are conformal
           mapping with  additional gluing
           properties.
           While it is easy to show their
           existence, see Subsection \ref{sub}, finding their
           expression is, generally speaking, quite a difficult task. Theorem \ref{ann} 
           of this paper, which
           gives explicit formulations for $w$ and $\widetilde{w}$,
           is one of our main contributions.
           Because of notations, we prefer stating it in Section \ref{Study_CGF} rather than here:
           the expressions of $w$ and $\widetilde{w}$ we obtain indeed strongly involve $\wp$-Weierstrass
           elliptic functions---the latter naturally appear due to the uniformization
           (that we make
           in Section \ref{Uniformization}) of
           the set given by the zeros of the kernel \eqref{def_kernel}, generically isomorphic to a Riemann surface
           of genus 1. In Theorems \ref{nature_CGF} and \ref{big_theo} below,
           we give important and complementary precisions to Theorem \ref{ann}.
\end{itemize}

For the 51 non-singular walks with an infinite group, the 
function $Q(x,y;z)$ is made explicit for the first time---to the
best of our knowledge. After the
work \cite{KRG} on Gessel's walk, this paper also provides integral representations of 
$Q(x,y;z)$ for the 22 other
walks admitting a finite group. As we are going to see now, it turns out
that the finiteness of the group actually acts directly on the nature
(rational, algebraic, holonomic, non-holonomic) of the functions $w$
and $\widetilde{w}$ present in these integral representations.
Another important quantity happens to be (the sign of) the covariance of
$\mathcal{S}$, namely,
     \begin{equation}
     \label{def_covariance}
          \sum_{(i,j)\in\mathcal{S}} ij.
     \end{equation}

\begin{thm}
\label{nature_CGF}
If the group of the walk is finite (resp.\ infinite), then $w$ and $\widetilde{w}$ are
algebraic (resp.\ non-holonomic, and then, of course, non-algebraic). Furthermore,
in the case of a finite group, if in addition the covariance \eqref{def_covariance}
of the walk is negative or zero (resp.\ positive), then $w$ and $\widetilde{w}$ are rational
(resp.\ algebraic non-rational).
\end{thm}

Theorem \ref{nature_CGF} is summarized in Figure \ref{Tables}: in particular,
in the finite group case, we compare \emph{a posteriori} the nature
of $w$ and $\widetilde{w}$ with that of $Q(x,y;z)$, known from \cite{BK,BMM,FR}.
If the group is finite, Theorem \ref{big_theo} below goes much further than Theorem \ref{nature_CGF}:
in the rational (resp.\ algebraic) case, it provides
the rational expressions (resp.\ the minimal polynomials) of
$w$ and $\widetilde{w}$.
For its statement we need the notations:
\begin{itemize}
\item[---]$x_3(z)$ and $x_4(z)$ (resp.\ $y_3(z)$ and $y_4(z)$) are the  remaining branch points
           of $Y_0$ and $Y_1$ (resp.\ $X_0$ and $X_1$), see again Subsection \ref{k} below \eqref{d_d_tilde}
           for their exact definition.
\end{itemize}

\newpage

\begin{thm}
\label{big_theo}
In the finite group case, the explicit expressions of $w$ and $\widetilde{w}$ are as follows.
\begin{enumerate}
\item \label{prop_CGF_order_4}
If the walk is associated with a group of order four,
then $w$ and $\widetilde{w}$ are affine combinations of, respectively,
     \begin{equation*}
          \frac{[t-x_{1}(z)][t-x_{4}(z)]}{[t-x_{2}(z)][t-x_{3}(z)]},
          \ \ \ \ \ \ \ \ \ \
          \frac{[t-y_{1}(z)][t-y_{4}(z)]}{[t-y_{2}(z)][t-y_{3}(z)]}.
     \end{equation*}
\item \label{prop_CGF_order_6_-}
For both walks $\{{\sf N},{\sf SE},{\sf W}\}$ and $\{{\sf N},{\sf E},{\sf SE},{\sf S},{\sf W},{\sf NW}\}$,
$w$ and $\widetilde{w}$ are affine combinations of, respectively,
     \begin{equation*}
     \label{CGF_order_6_-}
          \frac{u(t)}{[t-x_{2}(z)] [t-1/x_{2}(z)^{1/2}]^{2}},
          \ \ \ \ \ \ \ \ \ \
          \frac{\widetilde{u}(t)}{[t-y_{2}(z)] [t-1/y_{2}(z)^{1/2}]^{2}},
     \end{equation*}
with $u(t)=t^{2}$ and $\widetilde{u}(t)=t$ (resp.\ $u(t)=\widetilde{u}(t)=t(t+1)$) for
$\{{\sf N},{\sf SW},{\sf E}\}$ (resp.\ $\{{\sf N},{\sf W},{\sf SW},{\sf S},{\sf E},{\sf NE}\}$).
\item \label{prop_CGF_order_6_+}
For each of the three walks $\{{\sf NE},{\sf S},{\sf W}\}$, $\{{\sf N},{\sf E},{\sf SW}\}$
and $\{{\sf N},{\sf NE},{\sf E},{\sf S},{\sf SW},{\sf W}\}$,
there exist $\alpha(z),\beta(z),\delta(z),\gamma(z)$ which are algebraic with respect to $z$---and made explicit in
the proof---such that $w=\widetilde{w}$  is
the only root with a pole at $x_{2}(z)$ of
     \begin{equation*}
     \label{CGF_order_6_+}
          w^{2}+\left[\alpha(z)+ \frac{\beta(z)u(t)}{[t-x_{2}(z)][t-1/x_{2}(z)^{1/2}]^{2}}\right]w
          +\left[\delta(z)+\frac{ \gamma(z)u(t)}{[t-x_{2}(z)][t-1/x_{2}(z)^{1/2}]^{2}}\right],
     \end{equation*}
with $u(t)=t^{2}$ (resp.\ $u(t)=t$, $u(t)=t(t+1)$) for $\{{\sf NE},{\sf S},{\sf W}\}$ (resp.\ $\{{\sf N},{\sf E},{\sf SW}\}$,
$\{{\sf N},{\sf NE},{\sf E},{\sf S},{\sf SW},{\sf W}\}$).
\item \label{prop_CGF_order_8}
For the walk $\{{\sf E},{\sf SE},{\sf W},{\sf NW}\}$,
$w$ and $\widetilde{w}$ are affine
combinations of, respectively,
     \begin{equation*}
     \label{CGF_order_8}
          \frac{t^{2}}{[t-x_{2}(z)][t-1]^{2}[t-x_{3}(z)]},
          \ \ \ \ \ \ \ \ \ \
          \frac{t(t+1)^{2}}{[t-x_{2}(z)]^{2}[t-x_{3}(z)]^{2}}.
     \end{equation*}
\end{enumerate}
\end{thm}

We emphasize that it is \emph{not} necessary to know the affine combinations appearing in \ref{prop_CGF_order_4},
\ref{prop_CGF_order_6_-} and \ref{prop_CGF_order_8} above, since in Theorem \ref{explicit_integral},
$w$ and $\widetilde{w}$ only appear through $\partial_t w(t;z)/[w(t;z)-w(x;z)]$ and
$\partial_t \widetilde{w}(t;z)/[\widetilde{w}(t;z)-\widetilde{w}(y;z)]$.

Together with Gessel's walk $\{{\sf E},{\sf SW},{\sf W},{\sf NE}\}$,
for which the functions $w$ and $\widetilde{w}$ have been found in \cite{KRG},
Theorem \ref{big_theo} gives the simplified expression of these functions
in \emph{all} finite group cases, indeed see the tables in \cite{BMM} or Figure
\ref{Tables} in this paper.

Thanks to Theorems  \ref{big_theo} and \ref{ann},
which give formulations for $w$ and $\widetilde{w}$,
Theorem \ref{explicit_integral} settles the problem of
making explicit the numbers of walks $q(i,j;n)$. The
function $Q(x,y;z)$ being indeed found (at least) within the domain
$\{|x|\leq 1, |y|\leq 1, z\in ]0,1/k[\}$,
its coefficients $\sum_{n \geq 0} q(i,j;n)z^n$ can very easily be
obtained from the Cauchy formulas, for any  $z \in
  ]0,1/k[$. Since the radius of convergence of the  series
$\sum_{n \geq 0} q(i,j;n)z^n$ is at least $1/k$, the
  numbers of walks can then  be identified,
  \textit{e.g.}, in terms  of the limits
    of the successive derivatives\footnote{Throughout, for a function 
    $v$ we denote by $[v]^{(n)}$ or by $\partial^n v$ its $n$th derivative.} as $z>0$ goes to $0$:
    \begin{equation*}
         q(i,j;n)=\lim_{\substack{z\to 0\\(z\hspace{0.25mm}>\hspace{0.47mm}0)}}
         \frac{1}{n!}\left[\sum_{n \geq 0} q(i,j;n)z^n\right]^{(n)}.
    \end{equation*}
Other conclusions and perspectives of Theorems \ref{explicit_integral},
\ref{nature_CGF} and \ref{big_theo} are presented
in Section \ref{A_concluding_remark}.
Let us however make two remarks already.

\begin{rem}
{\rm Theorem \ref{explicit_integral} can be extended without difficulty to the case
of the generating function counting the numbers of paths confined to the 
quadrant $\mathbb{Z}_{+}^{2}$,
having length $k$, starting at $(i_{0},j_{0})$ and ending at $(i,j)$, for any initial
state $(i_{0},j_{0})$.}
\end{rem}

\begin{rem}{\rm\emph{A priori}, the expression of $Q(x,0;z)$
given in Theorem \ref{explicit_integral} is valid and holomorphic
for $x$ inside of some curve (see Subsection \ref{sub}) containing
the branch points $x_1(z)$ and $x_2(z)$\footnote{Each term of the sum providing 
$c(x;z)Q(x,0;z)-c(0;z)Q(0,0;z)$ in
Theorem \ref{explicit_integral} is clearly not holomorphic near
$[x_{1}(z),x_{2}(z)]$, but an application of the residue theorem
exactly as in \cite[Section 4]{KR} would give an expression of
the sum as a
function clearly holomorphic near $[x_{1}(z),x_{2}(z)]$.}.
This expression actually admits an analytic continuation on
$\mathbb{C}\setminus [x_3(z),x_4(z)]$: the arguments we gave for
proving \cite[Theorem 6]{KRG} for Gessel's walk can indeed be applied.
A similar remark holds for $Q(0,y;z)$.}
\end{rem}

Let us now turn to the $5$ singular walks. Both functions
$Q(x,0;z)$ and $Q(0,y;z)$ probably also admit integral
representations, see \cite[Part 6.4.1]{FIM}, but here we prefer
to give the following more elementary series representations.
Below, by $f^{\circ p}$ we mean
$f\circ \cdots \circ f$, with $p$ occurrences of $f$.
\begin{thm}
\label{explicit_series}
Suppose that the walk is singular. The following series representation holds:
     \begin{equation*}
     \label{series_representation}
          Q(x,0;z)=\frac{1}{z x^{2}}\sum_{p\geq 0} Y_{0}\circ (X_{0}\circ Y_{0})^{\circ p}(x;z)
          [(X_{0}\circ Y_{0})^{\circ p}(x;z) -(X_{0}\circ Y_{0})^{\circ (p+1)}(x;z)].
     \end{equation*}
The function $Q(0,y;z)$ is obtained from the equality above by exchanging the roles of $X_{0}$ and  $Y_{0}$.
Moreover, $Q(0,0;z)=0$, and the complete function $Q(x,y;z)$ is obtained with \eqref{functional_equation}.
\end{thm}

The paper \cite{MM2} gives a proof of Theorem \ref{explicit_series}
for the walks $\{{\sf NW},{\sf NE},{\sf SE}\}$ and $\{{\sf NW},{\sf
N},{\sf SE}\}$, and suggests that this result should also hold for the
other $3$ singular walks. For the sake of completeness, 
in this article we prove Theorem \ref{explicit_series} in
all the 5 cases, see Subsection \ref{Sing}.

%

\section{Approach via boundary value problems}
\label{apap}
\subsection{Structure of the remainder of the paper}
\label{gen_pre}

The approach that we are going to use depends on whether 
the walks under consideration are singular or not.

\medskip

In order to prove Theorem \ref{explicit_integral}---this concerns all the non-singular walks,
\textit{i.e.}, the 23 with a finite group
as well as 51 of the 56 attached to an infinite group,
see the tables in \cite{BMM}---we here
extend to three variables $x,y,z$ the profound analytic
approach elaborated by Fayolle, Iasnogorodski
and Malyshev in \cite{FIM} for the stationary case, which correspond to two
variables $x,y$. This is the aim of Subsection \ref{sub}. To summarize:

\medskip

\textit{Step 1.} \textit{From the fundamental functional equation
\eqref{functional_equation}, we prove that $c(x;z)Q(x,0;z)$ and
$\widetilde{c}(y;z)Q(0,y;z)$ satisfy certain boundary value
problems of Riemann-Carleman type, \textit{i.e.}, with boundary
conditions on curves closed in $\mathbb{C}\cup \{\infty\}$ and admitting non-empty interiors.}
These curves are studied in Lemma \ref{Properties_curves_0}.
The proof of this first step is performed in two stages. Firstly, we have to know the
class of functions within which $c(x;z)Q(x,0;z)$ and
$\widetilde{c}(y;z)Q(0,y;z)$ should be searched. This is the goal of
Theorem \ref{Thm_continuation}, which states that they admit nice
holomorphic continuations in the whole interiors of the
curves above. Its proof is postponed to Section \ref{CONTINUATION}.
Secondly, we have to obtain the precise boundary conditions on these
curves; this is done in Subsection \ref{sub}, see \eqref{BVP}.

\medskip

\textit{Step 2.} \textit{Next we transform these problems into
boundary value problems of Riemann-Hilbert type, \textit{i.e.}, with
conditions on segments}. This conversion is motivated by the fact 
that the latter problems
are more usual and by far more treated in the literature, see \cite{FIM,LIT} and references therein. 
It is done in Subsection \ref{sub}, by using
\textit{conformal gluing functions} in the sense of Definition
\ref{def_CGF}.

\medskip

\textit{Step 3.} \textit{Finally we solve these new problems and we
deduce an explicit integral representation of
\eqref{definition_generating_function}.} This will conclude
Subsection \ref{sub}.

\medskip

 Subsection \ref{sub} gives the complete proof of Theorem
\ref{explicit_integral}. As a consequence regarding the 
functions $c(x;z)Q(x,0;z)-c(0;z)Q(0,0;z)$ and
$\widetilde{c}(y;z)Q(0,y;z)-\widetilde{c}(0;z)Q(0,0;z)$, it remains
to find explicitly suitable conformal gluing functions $w$ and
$\widetilde{w}$. This is the topic of Sections \ref{Uniformization}
and \ref{Study_CGF}. In Section \ref{Study_CGF} we study 
these conformal gluing functions in-depth, as their analysis is just sketched
out in \cite{FIM}. There we state and prove Theorem \ref{ann}, which
  gives their explicit expression for all non-singular walks.
In Section \ref{Study_CGF} we also prove 
Theorem \ref{nature_CGF} on the nature of $w$ and $\widetilde{w}$, as well as Theorem \ref{big_theo}
giving simplified expressions of $w$ and $\widetilde{w}$ in the finite group case.
 The analysis of these functions relies on a uniformization of the
Riemann surface given by the zeros of the kernel,
namely, $\{(x,y)\in \mathbb{C}^{2}:$ $
\sum_{(i,j)\in\mathcal{S}}x^{i}y^{j}-1/z=0\}$. This work is carried out in
the introductory---and crucial---Section \ref{Uniformization}.

\medskip

For the 5 singular walks, the curves associated with the boundary value problems above
degenerate into a point and the previous arguments no longer work. However, starting from \eqref{functional_equation}, it is easy to make
explicit a series representation of the function \eqref{definition_generating_function},
see Subsection \ref{Sing}.

\subsection{The kernel and its roots}
\label{k} This part is introductory to Subsections \ref{sub} and
\ref{Sing}. It aims at examining the kernel \eqref{def_kernel} that
appears in \eqref{functional_equation}, and in particular at
studying its roots. First of all, we notice that the kernel can
alternatively be written as
     \begin{equation}
     \label{kernel_pt}
          xyz\left[\sum_{(i,j)\in\mathcal{S}}x^{i}y^{j}-1/z\right]=\widetilde{a}
          (y;z)x^{2}+\widetilde{b}(y;z)x+\widetilde{c}(y;z)=a(x;z)y^{2}+b(x;z)y+c(x;z),
     \end{equation}
where $\widetilde{c}(y;z)$ and $c(x;z)$ are defined in \eqref{c_c_tilde}, and where
     \begin{align*}
          \widetilde{a}(y;z)&=zy\textstyle \sum_{(+1,j)\in\mathcal{S}}y^{j},\hspace{-15mm}
          &\widetilde{b}(y;z)=-1+&zy\textstyle\sum_{(0,j)\in\mathcal{S}}y^{j}, \\
          a(x;z)&=\textstyle zx\sum_{(i,+1)\in\mathcal{S}}x^{i},\hspace{-15mm}
          &b(x;z)=-1+&zx\textstyle \sum_{(i,0)\in\mathcal{S}}x^{i}.
     \end{align*}
We also define
     \begin{equation}
     \label{d_d_tilde}
          \widetilde{d}(y;z)=\widetilde{b}(y;z)^{2}-4\widetilde{a}(y;z)\widetilde{c}(y;z),
          \ \ \ \ \  d(x;z)=b(x;z)^{2}-4a(x;z)c(x;z).
     \end{equation}
\begin{itemize}
  \item[---] If the walk is non-singular, then for any $z\in]0,1/k[$, the polynomial $\widetilde{d}$ (resp.\ $d$)
        has three or four roots, that we call $y_{k}(z)$ (resp.\ $x_{k}(z)$). As shown in \cite[Part 2.3]{FIM},
        they
        are such that $|y_{1}(z)|<y_{2}(z)<1<y_{3}(z)<|y_{4}(z)|$ (resp.\ $|x_{1}(z)|<x_{2}
        (z)<1<x_{3}(z)<|x_{4}(z)|$), with $y_4(z)=\infty$ (resp.\ $x_4(z)=\infty$)
        if $\widetilde{d}$ (resp.\ ${d}$) has order three.
  \item[---] If the walk is singular, the roots above then become $y_{1}(z)=y_{2}(z)=0<1<y_{3}(z)
        <|y_{4}(z)|$ (resp.\ $x_{1}(z)=x_{2}(z)=0<1<x_{3}(z)<|x_{4}(z)|$), see \cite[Part 6.1]{FIM}.
\end{itemize}
The behavior of the branch points $y_{k}(z)$ and $x_{k}(z)$ is not so simple
for $z\notin ]0,1/k[$, and for this reason \emph{we suppose in the sequel
that $z$ is fixed in $]0,1/k[$}.

Now we notice that the kernel \eqref{def_kernel} vanishes if and only if
$[\widetilde{b}(y;z)+ 2\widetilde{a}(y;z)x]^{2}=\widetilde{d}(y;z)$ or $[b(x;z)+2a(x;z)y
]^{2}=d(x;z)$. Consequently \cite{JS}, the algebraic functions $X(y;z)$ and $Y(x;z)$
defined by $\sum_{(i,j)\in\mathcal{S}}X(y;z)^{i}y^{j}-1/z=0$ and $\sum_{(i,j)
\in\mathcal{S}}x^{i}Y(x;z)^{j}-1/z=0$ have two branches, meromorphic on
$\mathbb{C}\setminus ([y_{1}(z),y_{2}(z)]\cup [y_{3}(z),y_{4}(z)])$ and
$\mathbb{C}\setminus([x_{1}(z),x_{2}(z)]\cup [x_{3}(z),x_{4}(z)])$
(resp.\ $\mathbb{C}\setminus [y_{3}(z),y_{4}(z)]$ and $\mathbb{C}\setminus
[x_{3}(z),x_{4}(z)]$) in the non-degenerate (resp.\ degenerate) case.

We fix the notations of the two branches of the algebraic
functions $X(y;z)$ and $Y(x;z)$ by setting  
     \begin{equation}
     \label{def_XX}
          X_{0}(y;z)=\frac{-\widetilde{b}(y;z)+ \widetilde{d}(y;z)^{1/2}}{2\widetilde{a}(y;z)},
          \ \ \ \ \ X_{1}(y;z)=\frac{-\widetilde{b}(y;z)-
          \widetilde{d}(y;z)^{1/2}}{2\widetilde{a}(y;z)},
     \end{equation}
as well as
     \begin{equation}
     \label{def_YY}
          Y_{0}(x;z)=\frac{-b(x;z)+d(x;z)^{1/2}}{2a(x;z)},
          \ \ \ \ \ Y_{1}(x;z)=\frac{-b(x;z)-d(x;z)^{1/2}}{2a(x;z)}.
     \end{equation}
The following straightforward result holds, see \cite[Part 5.3]{FIM}.
     \begin{lem}
     \label{Properties_X_Y_0}
          For all $y\in\mathbb{C}$, we have $|X_{0}(y;z)|
          \leq |X_{1}(y;z)|$. Likewise, for all $x\in\mathbb{C}$, we have $|Y_{0}(x;z)|
          \leq |Y_{1}(x;z)|$.
     \end{lem}

%
%

\subsection{Non-singular walks: proof of Theorem \ref{explicit_integral}}
\label{sub}

In this subsection we show Theorem \ref{explicit_integral},
dealing with all $74$ non-singular walks. According to
Subsection \ref{gen_pre}, we split the proof into three main steps.

\medskip

\textit{Step 1.}
We prove that both functions $c(x;z)Q(x,0;z)$ and
$\widetilde{c}(y;z)Q(0,y;z)$ satisfy a certain boundary
value problem of Riemann-Carleman type, with boundary conditions on
the curves
     \begin{equation*}
          X([y_{1}(z),y_{2}(z)];z),
          \ \ \ \ \ \ \ \ \ \ Y([x_{1}(z),x_{2}(z)];z),
     \end{equation*}
respectively. Examples of the latter are represented in Figure \ref{CurGe}.
In the general case, they satisfy the following properties, see \cite[Part 5.3]{FIM}.
     \begin{lem}
     \label{Properties_curves_0}
        Let $X([y_{1}(z),y_{2}(z)];z)$ and $Y([x_{1}(z),x_{2}(z)];z)$.
     \begin{itemize}
          \item[---] These two curves are symmetrical with respect to the real axis.
          \item[---] They are connected and closed in $\mathbb{C}\cup \{\infty\}$.
          \item[---] They split the plane into two connected components; we note
                    $\mathscr{G}X([y_{1}(z),y_{2}(z)];z)$ and $\mathscr{G}Y([x_{1}(z),x_{2}(z)];z)$
                    the ones of $x_{1}(z)$ and $y_{1}(z)$, respectively. They are such that
                    $\mathscr{G}X([y_{1}(z),y_{2}(z)];z)\subset\mathbb{C}\setminus
                    [x_{3}(z),x_{4}(z)]$ and $\mathscr{G}Y([x_{1}(z),x_{2}(z)];z)\subset
                    \mathbb{C}\setminus [y_{3}(z),y_{4}(z)]$.
     \end{itemize}
     \end{lem}

\begin{figure}[t]
\begin{center}
\begin{picture}(0.00,740.00)
\hspace{-117mm}
\includegraphics{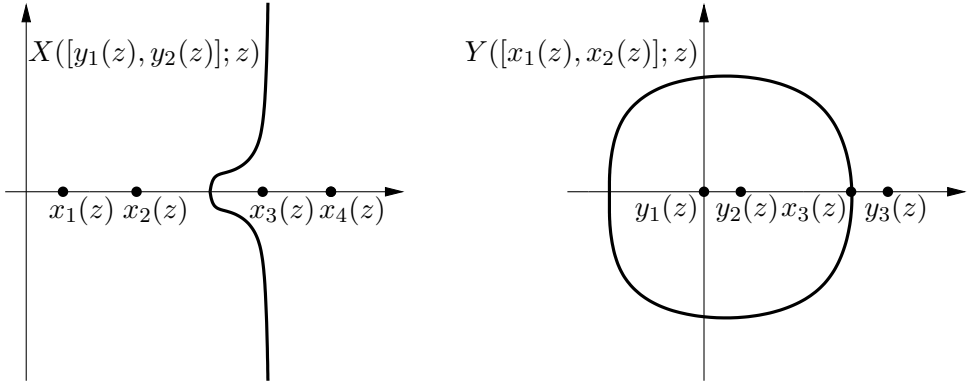}
\end{picture}
\end{center}
\vspace{-212mm}
\caption{The curves $X([y_{1}(z),y_{2}(z)];z)$ and $Y([x_{1}(z),x_{2}(z)];z)$ for Gessel's walk}
\label{CurGe}
\end{figure}

As illustrated by the example of Gessel's walk, see again Figure \ref{CurGe},
these curves are not always included in the unit disc, so that the functions $c(x;z)Q(x,0;z)$ and
$\widetilde{c}(y;z)Q(0,y;z)$ \textit{a priori} need not be defined on them. For this reason,
we first have to continue the generating functions up to these curves: this is exactly
the object of the following result, the proof of which is the subject of Section \ref{CONTINUATION}.
\begin{thm}
\label{Thm_continuation}
     The functions $c(x;z)Q(x,0;z)$ and $\widetilde{c}(y;z)Q(0,y;z)$ can be holomorphically
     continued from the open unit disc $\mathscr{D}$ to $\mathscr{G}X([y_{1}(z),y_{2}(z)]
     ;z)\cup \mathscr{D}$ and $\mathscr{G}Y([x_{1}(z),x_{2}(z)];z)\cup \mathscr{D}$, respectively.
\end{thm}
Now, exactly as in \cite[Section 1]{KRG}, we obtain:
     \begin{equation}
     \label{BVP}
          \hspace{-1mm}\left.\begin{array}{ccccc}
          \forall t\in X([\hspace{0.3mm}y_{1}(z),\hspace{0.3mm}y_{2}(z)];z) \hspace{-3.5mm}&, &\hspace{-2mm}
          c(t;z)Q(t,0;z)\hspace{-0.5mm}-\hspace{-0.5mm}
          c(\overline{t};z)Q(\overline{t},0;z)
          \hspace{-2.5mm}&=&\hspace{-2.5mm}
          tY_{0}(t;z)\hspace{-0.5mm}-\hspace{-0.5mm}\overline{t}Y_{0}(
          \overline{t};z),\\
          \forall t\in \hspace{0.3mm}Y([x_{1}(z),x_{2}(z)];z) \hspace{-3.5mm}&, &\hspace{-2mm}
          \widetilde{c}(t;z)Q(0,t;z)\hspace{-0.5mm}-\hspace{-0.5mm}
          \widetilde{c}(\overline{t};z)
          Q(0,\overline{t};z)\hspace{-2.5mm}&=&\hspace{-2.5mm}X_{0}(t;z)t
          \hspace{-0.5mm}-\hspace{-0.5mm}X_{0}
          (\overline{t};z)\overline{t}.
          \end{array}\right.
     \end{equation}
Together with Theorem \ref{Thm_continuation}, we thus conclude that
\emph{$c(x;z)Q(x,0;z)$ and $\widetilde{c}(y;z)Q(0,y;z)$ are found among the
functions holomorphic in  $\mathscr{G}X([y_{1}(z),y_{2}(z)];z)$ and $\mathscr{G}
Y([x_{1}(z),x_{2}(z)];z)$ and satisfying the conditions \eqref{BVP} on the boundary
of these sets}.

\medskip

\textit{Step 2.} A standard \cite{FIM,LIT} way  to solve these boundary value 
problems consists in converting them into problems with boundary
conditions on segments. This transformation is performed by the use of conformal gluing functions, defined below.

\begin{defn}
\label{def_CGF}
     Let $\mathscr{C}\subset\mathbb{C}\cup \{\infty\}$ be an open and simply connected set,
     symmetrical with respect to the real axis, and not equal to $\emptyset$, $\mathbb{C}$
     and $\mathbb{C}\cup \{\infty\}$. A function $w$ is said to be a conformal gluing
     function (CGF) for the set $\mathscr{C}$ if
     \begin{itemize}
     \item[---] $w$ is meromorphic in $\mathscr{C}$;
     \item[---] $w$ establishes a conformal
     mapping of $\mathscr{C}$ onto the complex plane cut along a segment;
     \item[---] For all $t$ in the boundary of $\mathscr{C}$, $w(t)=w(\overline{t})$.
     \end{itemize}
\end{defn}

\begin{rem}{\rm It is worth noting that the existence (but without any explicit expression) of a CGF
for a generic set $\mathscr{C}$ is ensured by general results on conformal gluing
\cite[Chapter 2]{LIT}.}\end{rem}

Let $w(t;z)$ be a CGF for the set $\mathscr{G}X([y_{1}(z),y_{2}(z)];z)$ and let
$U_z$ denote the real segment $w(X([y_{1}(z),y_{2}(z)];z);z)$. We also define
$v(u;z)$ as the reciprocal function of $w(t;z)$. It is meromorphic on
$\mathbb{C}\setminus U_z$, see Definition \ref{def_CGF}. Finally, 
for real values of $u$ we set
$v^{+}(u;z)$ (resp.\ $v^-(u;z)$) for the limit of $v(s;z)$ as $s\to u$ from
the upper (resp.\ lower) half-plane. We notice that for $u\in U_z$,
$v^{+}(u;z)$ and $v^{-}(u;z)$ are different and complex conjugate of one another.

With the notations above, the boundary value problem of the first step
becomes that of
\textit{finding a function $c(v(u;z);z)Q(v(u;z),0;z)$ holomorphic
in $\mathbb{C}\setminus U_z$, bounded near the ends of $U_z$, and
such that, for $u\in U_z$,}
     \begin{multline*}
          c(v^+(u;z);z)Q(v^+(u;z),0;z)-c(v^-(u;z);z)Q(v^-(u;z),0;z)=\\
          v^+(u;z)Y_0(v^+(u;z);z)-v^-(u;z)Y_0(v^-(u;z);z).
     \end{multline*}

     \medskip

\textit{Step 3.}
This problem is standard \cite{FIM,LIT} and can be immediately resolved.
It yields, up to an additive function of $z$,
     \begin{equation*}
               c(v(u;z);z)Q(v(u;z),0;z)=
               \frac{1}{2\pi \imath}\int_{U_z}
               \frac{v^+(s;z)Y_0(v^+(s;z);z)-v^-(s;z)Y_0(v^-(s;z);z)}{s-u}\text{d}s.
     \end{equation*}
The change of variable $s=w(t;z)$ then gives, up to an additive function of $z$,
          \begin{equation}
          \label{btc}
               c(x;z)Q(x,0;z)=
               \frac{1}{2\pi \imath}\int_{X([y_{1}(z),y_{2}(z)];z)}
               t Y_0(t)\frac{\partial_t w(t;z)}{w(t;z)-w(x;z)}\text{d}t.
          \end{equation}
We note that for the difference $c(x;z)Q(x,0;z)-c(0;z)Q(0,0;z)$
considered in Theorem \ref{explicit_integral}, it is \emph{not} worth
knowing the additive function of $z$ appearing in \eqref{btc}. The
residue theorem applied exactly
 as in \cite[Section 4]{KRG}
then transforms the integral on a curve \eqref{btc} into the
integral on a segment written in Theorem \ref{explicit_integral},
 the proof of the formula for $c(x;z)Q(x,0;z)-c(0;z)Q(0,0;z)$ is completed.

The expression of $\widetilde{c}(y;z)Q(0,y;z)-\widetilde{c}(0;z)Q(0,0;z)$ is derived similarly. 
Suppose now that ${\sf SW}\notin \mathcal{S}$, or equivalently that $\widetilde{c}(0;z)=0$. The last formula
then yields an expression of $\widetilde{c}(y;z)Q(0,y;z)$, whence of $Q(0,y;z)$ by division, and finally
of $Q(0,0;z)$ by substitution. If now ${\sf SW}\in \mathcal{S}$, evaluating the functional equation \eqref{functional_equation}
at any $(x_0,y_0,z_0)\in\{|x|\leq 1, |y|\leq 1, |z|<1/k\}$ at which the
kernel \eqref{def_kernel} vanishes immediately provides the expression of $Q(0,0;z)$ stated in Theorem
\ref{explicit_integral}. As for $Q(x,y;z)$, it is then sufficient to use the functional
equation.

\subsection{Singular walks: proof of Theorem \ref{explicit_series}}
\label{Sing}

Theorem \ref{explicit_series} is shown in \cite{MM2} for both step
sets $\mathcal{S}= \{{\sf NW},{\sf NE},{\sf SE}\}$ and $\{{\sf
NW},{\sf N},{\sf SE}\}$. In this subsection we explain how to obtain
it for all $5$ singular walks. Along the same lines as in
\cite[Part 6.4]{FIM}, we obtain from \eqref{functional_equation} the
identity
     \begin{equation}
     \label{toit}
          c(X_0\circ Y_0(t;z);z)Q(X_0\circ Y_0(t;z),0;z)-c(t;z)Q(t,0;z)=X_0\circ Y_0(t;z)Y_0(t;z)
          -tY_0(t;z).
     \end{equation}
Applying \eqref{toit} for $t=(X_{0}\circ Y_{0})^{\circ p}(x;z)$ and
summing with respect to $p\in\{0,\ldots ,q\}$ formally gives, with
the same notations as in the statement of Theorem
\ref{explicit_series},
     \begin{multline*}
          c((X_{0}\circ Y_{0})^{\circ (q+1)}(x;z);z)Q((X_{0}\circ Y_{0})^{\circ (q+1)}(x;z),0;z)-c(x;z)Q(x,0;z)=\\
          \sum_{p= 0}^q Y_{0}\circ (X_{0}\circ Y_{0})^{\circ p}(x;z)
          [(X_{0}\circ Y_{0})^{\circ p}(x;z) -(X_{0}\circ Y_{0})^{\circ (p+1)}(x;z)].
     \end{multline*}
Using \cite[Part 6.4.2]{FIM}, \textit{i.e.}, that (at least)
for $|x|<1$ and $|z|<1/k$, $|X_{0}\circ Y_{0}(x;z)|<x$, gives that
the series above is convergent as $q\to\infty$. This also implies 
that the left-hand side member goes to $c(0;z)Q(0,0;z)-c(x;z)Q(x,0;z)=-zx^{2}Q(x,0;z)$ 
as $q\to\infty$, since for all $5$ singular walks $c(x;z)=zx^2$, see
\eqref{c_c_tilde}. Theorem \ref{explicit_series} follows.

\section{Holomorphic continuation of the generating functions}
\label{CONTINUATION}

\medskip

This part aims at proving Theorem \ref{Thm_continuation}. In other words, we must show that
the generating functions $c(x)Q(x,0)$\footnote{For the sake of conciseness we will, from now on, drop the dependence of
the different quantities on $z\in]0,1/k[$. Moreover, it is implied that any statement in the sequel
begins with ``for any $z\in]0,1/k[$''.} and $\widetilde{c}(y)Q(0,y)$, already
known to be holomorphic in their open unit disc $\mathscr{D}$, can be holomorphically continued
up to $\mathscr{G} X([y_{1},y_{2}])\cup \mathscr{D}$ and $\mathscr{G}Y([x_{1},x_{2}])\cup \mathscr{D}$,
respectively.
First of all, note that the location of the sets $\mathscr{G}X([y_{1},y_{2}])$ and $\mathscr{G}
Y([x_{1},x_{2}])$ depends strongly on the step set $\mathcal{S}$. In particular, it may happen
that they are included in the unit disc---\textit{e.g.}, it is the case for the walks
$\{{\sf N},{\sf SE},{\sf W}\}$ and $\{{\sf N},{\sf E},{\sf SE},{\sf S},{\sf W},{\sf NW}\}$,
as it will be further illustrated in Figure \ref{Ex_Pa}---and in that case Theorem
\ref{Thm_continuation} is obvious. On the other hand, there actually exist walks for which these sets
do not lie inside the unit disc---this is true for Gessel's walk, see Figure \ref{CurGe}.
The proof of Theorem \ref{Thm_continuation} requires us the following results.

\begin{lem}
Assume that the walk is non-singular. The following properties hold:
\label{Delta_connected}
\begin{enumerate}
\item \label{Y011}
$Y_{0}(\{|x|=1\})\subset \{|y|<1\}$;
\item \label{lemma_non_singular_reciprocal}
$X_{0}:\mathscr{G}Y([x_{1},x_{2}])\setminus [y_{1},y_{2}]\to
\mathscr{G}X([y_{1},y_{2}])\setminus [x_{1},x_{2}]$ and $Y_{0}: \mathscr{G}X([
y_{1},y_{2}])\setminus [x_{1},x_{2}]\to\mathscr{G}Y([x_{1},x_{2}])\setminus [y_{1},
y_{2}]$ are conformal and reciprocal of one another;
\item \label{fifi} $\{x\in\mathbb{C}: |Y_{0}(x)|< 1\}\cap \mathscr{D}$ is non-empty;
\item \label{se} $\mathscr{G}X([y_{1},y_{2}])\cup \mathscr{D}$ is connected;
\item \label{th} $\mathscr{G}X([y_{1},y_{2}])\cup\mathscr{D}$ is included in
$\{x\in\mathbb{C}: |Y_{0}(x)|<1\}\cup \mathscr{D}$.
\end{enumerate}
\end{lem}

\begin{proof}
Let us recall that $Y_{0}$ is one of the two
$y$-roots of the kernel \eqref{def_kernel},
and that with $Y_{1}$
denoting the other one, we have $|Y_{0}|\leq |Y_{1}|$, see Lemma \ref{Properties_X_Y_0}.
We are going to prove \ref{Y011} first for $z=1/k$, and we shall deduce from this
the remaining cases $z\in]0,1/k[$.
\begin{itemize}
\item[---] Assume that $z=1/k$. If the kernel \eqref{def_kernel} vanishes at $(x,y$) then
$\sum_{(i,j)\in \mathcal{S}}(1/k)x^{i}y^{j}=1$. Since $\sum_{(i,j)\in \mathcal{S}}(1/k)=1$,
we can apply \cite[Lemma 2.3.4]{FIM}, and in this way we  obtain that $Y_{0}
(\{|x|=1\})\subset\{|y|\leq 1\}$.
\item[---] Suppose that $z\in]0,1/k[$. In that case, the kernel cannot vanish at $(x,y)$
with $|x|=|y|=1$: indeed, for $|x|=|y|=1$ we clearly have that $|\sum_{(i,j)
\in\mathcal{S}}x^{i}y^{j}|\leq \sum_{(i,j)\in \mathcal{S}}1=k <1/z$. As a consequence,
$Y_{0}(\{|x|=1\})\cap\{|y|= 1\}$ is empty. By connectedness, this implies that either
$Y_{0}(\{|x|=1\})\subset \{|y|<1\}$ or $Y_{0}(\{|x|=1\})\subset \{|y|>1\}$.
But once again with \cite[Lemma 2.3.4]{FIM}, for $z=1/k$ we obtain $Y_{0}(\{|x|=1\})\cap\{|y|< 1\}
\neq \emptyset$, so that by continuity
we get that $Y_{0}(\{|x|=1\})\subset \{|y|<1\}$ for all $z\in]0,1/k[$.
\end{itemize}
Item \ref{lemma_non_singular_reciprocal} is proved in \cite[Part 5.3]{FIM} for $z=1/k$; the proof for other values
of $z$ is similar and we therefore choose to omit it.
Note now that \ref{fifi} is a straightforward consequence of \ref{Y011}.
Item \ref{se} is also clear: both sets $\mathscr{G}X([y_{1},y_{2}])$ and $\mathscr{D}$
are connected and the intersection $\mathscr{G}X([y_{1},y_{2}])\cap \mathscr{D}$
is non-empty, since $x_{1}$ belongs to both
sets.
For \ref{th}, it is enough to prove that $(\mathscr{G}X([y_{1},y_{2}])
\cup \mathscr{D})\setminus \mathscr{D}$ is included in $\{x\in\mathbb{C}: |Y_{0}(x)|< 1\}$.
This will follow from an application of the maximum modulus principle (see, \textit{e.g.}, \cite{JS})
to the function $Y_{0}$ on $(\mathscr{G}X([y_{1},y_{2}])\cup \mathscr{D})\setminus \mathscr{D}$.
First of all let us note that $Y_{0}$ is analytic on the latter domain, since thanks to
Subsections \ref{k} and \ref{sub} it is included in $\mathbb{C}\setminus ([x_{1},x_{2}]\cup[x_{3},x_{4}])$.
Next, we prove that $|Y_{0}|< 1$ on the boundary of the set $(\mathscr{G}X([y_{1},y_{2}])\cup
\mathscr{D})\setminus \mathscr{D}$, and for this it is sufficient to show that $|Y_{0}|< 1$ on
$\{|x|= 1\}\cup X([y_{1},y_{2}])$.
But by \ref{Y011} it is immediate that $|Y_{0}|< 1$ on $\{|x|=1\}$, and by
\ref{lemma_non_singular_reciprocal} we get that $Y_{0}(X([y_{1},y_{2}]))=[y_{1},y_{2}]$, a
segment which is known to belong to the unit disc, thanks to Subsection \ref{k}.
In this way, the maximum modulus principle directly entails that $|Y_{0}|< 1$ on
the domain $(\mathscr{G}X([y_{1},y_{2}])\cup \mathscr{D})\setminus \mathscr{D}$.
\end{proof}

\begin{proof}[Proof of Theorem \ref{Thm_continuation}.]
We are going to explain here the continuation procedure only for $c(x)Q(x,0)$,
since the case of $\widetilde{c}(y)Q(0,y)$ is similar.

Evaluating \eqref{functional_equation} at any $(x,y)\in\{|x|\leq 1,|y|\leq 1\}$
such that $\sum_{(i,j)\in\mathcal{S}}x^{i}y^{j}-1/z=0$ leads to
$c(x)Q(x,0)+\widetilde{c}(y)Q(0,y)-z\delta Q(0,0)-xy=0$. Therefore, if $x\in \{x \in
\mathbb{C}: |Y_{0}(x)|< 1\}\cap \mathscr{D}$, we obtain
     \begin{equation}
     \label{before_cont}
          c(x)Q(x,0)+\widetilde{c}(Y_{0}(x))Q(0,Y_{0}(x))-z\delta Q(0,0)-xY_{0}(x)=0.
     \end{equation}
Since $\{x \in \mathbb{C}: |Y_{0}(x)|< 1\}\cap \mathscr{D}$ is non-empty,
see Lemma \ref{Delta_connected} \ref{fifi}, both functions $c(x)Q(x,0)$ and
$\widetilde{c}(Y_{0}(x))Q(0,Y_{0}(x))$ as well as the identity \eqref{before_cont}
can be extended up to the connected component of $\{x\in\mathbb{C}: |Y_{0}(x)|
< 1\}\cup \mathscr{D}$ containing $\{x \in \mathbb{C}: |Y_{0}(x)|< 1\}\cap
\mathscr{D}$, by analytic continuation.

As a consequence, we now need to show that $\mathscr{G}X([y_{1},y_{2}])\cup \mathscr{D}$ is connected
and included in $\{x\in\mathbb{C}: |Y_{0}(x)|< 1\}\cup \mathscr{D}$,
since $\mathscr{G}
X([y_{1},y_{2}])\cup \mathscr{D}$ has clearly a non-empty intersection
with $\{x \in \mathbb{C}: |Y_{0}(x)|< 1\}\cap \mathscr{D}$. These facts
are exactly the objects of Lemma \ref{Delta_connected} \ref{se} and \ref{th}.

It thus remains for us only to prove that this continuation of $c(x)Q(x,0)$ is
holomorphic on $\mathscr{G}X([y_{1},y_{2}])\cup \mathscr{D}$.
\begin{itemize}
\item[---] On $\mathscr{D}$ this is immediate, $c(x)Q(x,0)$ being therein defined by its power series.
\item[---] On $(\mathscr{G}X([y_{1},y_{2}])\cup \mathscr{D})\setminus \mathscr{D}$, it follows from \eqref{before_cont}
that the function $c(x)Q(x,0)$ may possibly have the same singularities as $Y_{0}$---namely, the branch
cuts $[x_{1},x_{2}],[x_{3},x_{4}]$---and is holomorphic elsewhere. But these segments do not belong to
$(\mathscr{G}X([y_{1},y_{2}])\cup \mathscr{D})\setminus \mathscr{D}$: with Subsection \ref{k}
we have that $[x_{1},x_{2}]$ is included in $\mathscr{D}$, and by Lemma \ref{Properties_curves_0} we obtain that
$[x_{3},x_{4}]$ is exterior to $\mathscr{G}X([y_{1},y_{2}])$.
\end{itemize}
The continuation of $c(x)Q(x,0)$ is thus holomorphic on $\mathscr{G}X([y_{1},y_{2}])
\cup \mathscr{D}$ and Theorem \ref{Thm_continuation} is proved.
\end{proof}

Let us note that in \cite{KRG},
 we also introduced a procedure of continuation
of $c(x)Q(x,0)$ and $\widetilde{c}(y)Q(0,y)$. We have chosen to present here another way:
it is weaker since Theorem \ref{Thm_continuation}
yields a continuation of the generating functions up to $\mathscr{G}X([y_{1},y_{2}])\cup \mathscr{D}$
and $\mathscr{G}Y([x_{1},x_{2}])\cup \mathscr{D}$ and not up to $\mathbb{C}\setminus [x_{3},x_{4}]$
and $\mathbb{C}\setminus [y_{3},y_{4}]$ as in \cite{KRG};  it is, however, more elementary, because this continuation
is performed directly on the complex plane, rather than on  a uniformization of
the set given by the zeros of the kernel \eqref{def_kernel} as in \cite{KRG}.



\section{Uniformization}
\label{Uniformization}

This part serves to introduce notions that are crucial to Section \ref{Study_CGF}.
It amounts to studying closely the set of zeros of the kernel \eqref{def_kernel},
namely,
     \begin{equation*}
          \mathscr{K}=\{(x,y)\in \mathbb{C}^{2}: \textstyle\sum_{(i,j)\in\mathcal{S}}x^{i}y^{j}-1/z=0\}.
     \end{equation*}
     \begin{prop}
     \label{genus_Riemann_surface}
          For any non-singular walk, $\mathscr{K}$ is a Riemann surface of genus one.
     \end{prop}
\begin{proof}
Subsection \ref{k} yields that $\sum_{(i,j)\in\mathcal{S}}x^{i}y^{j}-1/z=0$
if and only if $[b(x)+2a(x)y]^{2}=d(x)$. But the Riemann surface of the square root of a
third or fourth-degree polynomial with distinct roots has genus one, see, \textit{e.g.}, \cite{JS}.
Therefore the genus of $\mathscr{K}$ is also one.
\end{proof}

\begin{rem}
{\rm Note that Proposition \ref{genus_Riemann_surface} cannot be
extended to the singular walks. Indeed, it follows from Subsection
\ref{k} that for these walks, the polynomial $d$
has a double root at $0$ and two simple roots at $x_{3}$ and $x_{4}$,
and it is well known, see \cite{JS}, that the Riemann
surface of the square root of such a polynomial has genus zero.
We also notice that Proposition \ref{genus_Riemann_surface}, implicitly
stated for $z\in]0,1/k[$, cannot be extended to $z=0$ or $z=1/k$ in
the general case.
Indeed, for $z=0$ we have $d(x)=x^{2}$ and the Riemann surface
of the square root of this polynomial is a disjoint union of
two spheres, see \cite{JS}.
For $z=1/k$, it may happen that the genus of $\mathscr{K}$
is still one, but it may also happen that it becomes zero. In
fact, \cite[Part 2.3 and Part 6.1]{FIM} entails that it equals zero
if and only if the two  equalities
$\sum_{(i,j)\in\mathcal{S}}i=0$ and $\sum_{(i,j)\in \mathcal{S}}j=0$ hold.}
\end{rem}

With Proposition \ref{genus_Riemann_surface}, it is immediate that $\mathscr{K}$ is
isomorphic to a certain torus $\mathbb{C}/\Omega$. A suitable lattice $\Omega$---in fact
the \emph{only possible} lattice, up to a homothetic transformation---is made explicit
in \cite[Part 3.1 and Part 3.3]{FIM}, namely, $\Omega=\omega_{1}\mathbb{Z}+\omega_{2}
\mathbb{Z}$, with
     \begin{equation}
     \label{def_omega_1_2}
          \omega_{1}= \imath\int_{x_{1}}^{x_{2}}
          \frac{\text{d}x}{[-d(x)]^{1/2}},
          \ \ \ \ \ \omega_{2}= \int_{x_{2}}^{x_{3}}
          \frac{\text{d}x}{[d(x)]^{1/2}}.
     \end{equation}
Using the same arguments as in \cite[Part 3.3]{FIM}, we immediately obtain
in addition the \emph{uniformization}
     \begin{equation*}
          \mathscr{K}=\{(x(\omega),y(\omega)),\omega\in\mathbb{C}/\Omega\},
     \end{equation*}
with
     \begin{equation}
     \label{uniformization}
          x(\omega)=F(\wp(\omega),\wp'(\omega)),
          \ \ \ \ \ y(\omega)=G(\wp(\omega),\wp'(\omega)),
     \end{equation}
where $F(p,p')=x_{4}+d'(x_{4})/[p-d''(x_{4})/6]$ and $G(p,p')=[-b(F(p,p'))+d'(x_{4})
p'/(2[p-d''(x_{4})/6]^{2})]/[2a(F(p,p'))]$ if $x_{4}\neq \infty$, while
$F(p,p')=[6p-d''(0)]/d'''(0)$ and $G(p,p')=[-b(F(p,p'))-3p'/d'''(0)]/[2a(F(p,p'))]$
if $x_{4}=\infty$,
and where $\wp$ denotes the Weierstrass elliptic function with periods $\omega_{1},\omega_{2}$.
By definition, see \cite{JS,WW},  $\wp$ is equal to
     \begin{equation*}
          \wp(\omega)=\frac{1}{\omega^2}+\sum_{(p_1,p_2)\in\mathbb{Z}^{2}\setminus (0,0)}
          \left[\frac{1}{(\omega-p_1\omega_1-p_2\omega_2)^{2}}-
          \frac{1}{(p_1\omega_1+p_2\omega_2)^{2}}\right],
     \end{equation*}
     and it is well known that it satisfies the differential equation
     \begin{equation}
     \label{diff_eq_wp}
     \wp'(\omega)^{2}=4[\wp(\omega)-\wp(\omega_{1}/2)]
                       [\wp(\omega)-\wp([\omega_{1}+\omega_{2}]/2)]
                       [\wp(\omega)-\wp(\omega_{2}/2)].
     \end{equation}

The main motivation for introducing a uniformization of $\mathscr{K}$ is to render
the roles of $x$ and $y$ more symmetric. This is why we now go one
step further, and we look for the location on $\mathbb{C}$, or equivalently
on the fundamental parallelogram $[0,\omega_{2}[
\times $ $[0,\omega_{1}/\imath[$, of the reciprocal images through the uniformization
\eqref{uniformization} of the
important cycles that are the branch cuts $[x_{1},x_{2}]$, $[x_{3},x_{4}]$, $[y_{1},y_{2}]$
and $[y_{3},y_{4}]$. For the first two, we introduce
     \begin{equation}
     \label{def_f}
          f(t)=\left\{\begin{array}{lll}
          \displaystyle d''(x_{4})/6+d'(x_{4})/[t-x_{4}]& \text{if} & x_{4}\neq \infty,\\
          \displaystyle d''(0)/6+d'''(0)t/6 \phantom{{1^1}^{1}}& \text{if} & x_{4}=\infty.\end{array}\right.
     \end{equation}
It is such that $x(\omega)=f^{-1}(\wp(\omega))$, see \eqref{uniformization}. Further, it follows from the
construction that $\wp(\omega_{1}/2)=f(x_{3})$, $\wp([\omega_{1}+\omega_{2}]/2)=f(x_{2})$ and $\wp(\omega_{1}/2)=f(x_{1})$;
for a proof, see \cite[Part 3.3]{FIM}. As for $[y_{1},y_{2}]$
and $[y_{3},y_{4}]$, we need to introduce a new period, namely,
     \begin{equation}
     \label{def_omega_3}
          \omega_{3}= \int_{X(y_{1})}^{x_{1}}
          \frac{\text{d}x}{[d(x)]^{1/2}}.
     \end{equation}
In \cite[Part 3.3]{FIM}, the following is shown.
     \begin{lem}
     \label{lem_omega_2>3}
          The period $\omega_{3}$ lies in $]0,\omega_{2}[$.
     \end{lem}
By using the same analysis of \cite[Part 5.5]{FIM}, we
obtain the following pleasing result.
     \begin{prop}
     \label{transformation_cycles_uniformization}
     We have
     \begin{align*}
     x^{-1}([x_{1},x_{2}])&=[0,\omega_{1}[+\omega_{2}/2,
     &x^{-1}([x_{3},x_{4}])&=[0,\omega_{1}[,\\
     y^{-1}([y_{1},y_{2}])&=[0,\omega_{1}[+(\omega_{2}+\omega_{3})/2,
     &y^{-1}([y_{3},y_{4}])&=[0,\omega_{1}[+\omega_{3}/2.
     \end{align*}
     \end{prop}
In our opinion, the latter result justifies the introduction of the
uniformization: indeed, on the fundamental parallelogram, the quartic curves
$X([y_1,y_2])$ and $Y([x_1,x_2])$ just become segments! Proposition
\ref{transformation_cycles_uniformization} in particular asserts that
$\omega_3$ naturally appears when locating the branch cuts on the
uniformization space. In fact, this period also plays a crucial role with
respect to the group $W=\langle \Psi,\Phi\rangle$ defined in \eqref{def_generators_group}.
Indeed, the two birational transformations $\Psi$ and $\Phi$ of $\mathbb{C}^{2}$ can \emph{a fortiori} be
understood as automorphisms of $\mathscr{K}$, and we recall \cite[Part 3.1]{FIM}
that thanks to \eqref{uniformization}, these automorphisms of
$\mathscr{K}$ become on $\mathbb{C}/\Omega$ the automorphisms
$\psi$ and $\phi$ with the following expressions:
     \begin{equation}
     \label{expression_automorphisms_torus}
          \psi(\omega)=-\omega,\ \ \ \ \
          \phi(\omega)=-\omega+\omega_{3}.
     \end{equation}
They satisfy $\psi\circ\psi=\phi\circ\phi=\text{id}$, $x\circ\psi=x$,
$y\circ\psi=[c(x)/a(x)]/y$, $x\circ\phi=[\widetilde{c}(y)/\widetilde{a}(y)]/x$ and
$y\circ\phi=y$.

Our goal now is to find a characterization of the finiteness
of the group in terms of $\omega_3$. We shall obtain it in Proposition \ref{lemma_omega_2_3_finite}.
The latter requires us the preliminary result hereunder.

     \begin{lem}
     \label{lemma_omega_2>=<3}
          We have $\omega_{3}< \omega_{2}/2$
          (resp. $\omega_{3}=\omega_{2}/2$, $\omega_{3}>\omega_{2}/2$) if and only
          if the covariance \eqref{def_covariance}
          of the walk is negative (resp.\ zero, positive).
     \end{lem}
\begin{proof}
We just give here the sketch of the proof; for the details, we refer to
\cite[Section 4]{KR}. The first step consists in noticing that $\omega_{3}=\omega_{2}/2$
if and only if the covariance equals zero, which follows from showing that
both assertions are equivalent to having a group of order $4$. This essentially
implies that it suffices to find \emph{one} walk
for which we simultaneously have $\omega_3<\omega_2/2$ (resp.\ $\omega_3<\omega_2/2$)
and a negative (resp.\ positive) covariance. This verification is performed
in \cite{KR} for the walk $\{{\sf E},{\sf S},{\sf NW}\}$ (resp.\ $\{{\sf NE},{\sf S},{\sf W}\}$).
\end{proof}
It could appear surprising to introduce the covariance here. As
we will see throughout, its sign actually strongly influences the
behavior of many quantities, see Theorem \ref{nature_CGF}, itself summarized in Figure \ref{Tables}.
If the group $W=\langle \Psi,\Phi\rangle$ is finite,
we can give the following important precisions to Lemma \ref{lemma_omega_2>=<3}.
     \begin{prop}
     \label{general_lemma_omega_2_3_finite}
          For any $k\geq 2$, the group $W=\langle \Psi,\Phi\rangle$ has order
          $2k$ if and only if there exists an integer $q\in\{1,\ldots ,k-1\}$,
          independent of $z$ and having no common divisors with $k$, such that
          for all $z\in]0,1/k[$, $\omega_{3}=(q/k)\omega_{2}$.
     \end{prop}

\begin{proof}
We first prove the converse sense
and for this purpose, we reintroduce the variable
$z$ in the notations.

Assume first that for  \emph{some} value $z\in]0,1/k[$, one has $\omega_{3}/\omega_{2}=q/k$,
where the integers $q$ and $k$ have no common divisors. This means that $\langle \psi,\phi
\rangle$ is a finite group of order $2k$, see \eqref{expression_automorphisms_torus},
which is equivalent to the fact that $\langle \Psi,\Phi\rangle$ has order $2k$ on
$\mathscr{K}$---but
\emph{a priori} not on $\mathbb{C}^{2}$. This implies that for any $x\in\mathbb{C}$,
the equalities $(\Phi\circ\Psi)^{\circ k}(x,Y_{0}(x;z))=(x,Y_{0}(x;z))$ and
$(\Phi\circ\Psi)^{\circ k}(x,Y_{1}(x;z))=(x,Y_{1}(x;z))$ hold.

Suppose now that for \emph{any} value $z\in]0,1/k[$, the quantity
$\omega_{3}/\omega_{2}$ is this same rational number $q/k$.
In particular for any fixed $x\in \mathbb{C}$ and all $z\in]0,1/k[$,
$(\Phi\circ\Psi)^{\circ k}(x,Y_{0}(x;z))=(x,Y_{0}(x;z))$. In other
words, for any fixed $x\in \mathbb{C}$ and any $y\in\{Y_{0}(x;z):
z\in ]0,1/k[\}$, $(\Phi\circ\Psi)^{\circ k}(x,y)=(x,y)$. Since the
set $\{Y_{0}(x;z): z\in ]0,1/k[\}$ is not isolated, by analytic
continuation we obtain that for any $x\in \mathbb{C}$ and any $y\in
\mathbb{C}$, $(\Phi\circ\Psi)^{\circ k}(x,y)=(x,y)$, in such a way
that $\langle \Psi,\Phi\rangle $ is finite as a group of
birational transformations of $\mathbb{C}^{2}$, of order less than or equal to $2k$.

The group $\langle \Psi,\Phi\rangle $ is of order exactly $2k$ because $\langle
\psi,\phi\rangle$ has order $2k$ and thus we can find some elements
$(x,y)$ of order exactly $2k$. This entails the converse sense of
Proposition \ref{general_lemma_omega_2_3_finite}.

Suppose now that $W=\langle \Psi,\Phi\rangle$ has order $2k$. The group
$\langle \psi,\phi\rangle$ generated by $\psi$ and $\phi$ is \textit{a
fortiori}  finite, of order $2r(z)\leq 2k$, which means that $\inf \{p>0:
(\phi\circ\psi)^{\circ p}=\text{id}\}=r(z)$. With \eqref{expression_automorphisms_torus}
this immediately implies that $r(z)\omega_{3}$ is some point of the lattice $\omega_{1}\mathbb{Z}+
\omega_{2}\mathbb{Z}$, contrary to $p\omega_{3}$ for $p\in\{1,\ldots ,r(z)-1\}$.
But with Lemma \ref{lem_omega_2>3} we have $\omega_{3}\in]0,\omega_{2}[$, so
that we get $r(z)\omega_{3}=q(z)\omega_{2}$, where $q(z)\in\{1,\ldots ,r(z)-1\}$
has no common divisors with $r(z)$.

Moreover, thanks to \eqref{def_omega_1_2} and \eqref{def_omega_3} we know
that $\omega_{3}/\omega_{2}=q(z)/r(z)$ is a holomorphic function of $z$;
taking rational values it has to be constant, say $\omega_{3}/\omega_{2}=q/r$.
Finally, if $r$ was strictly smaller than $k$, then with the first part of the
proof we would obtain that $W=\langle \Psi,\Phi\rangle$ has also order $2r$,
and not $2k$ as assumed, so that $r=k$ and Proposition \ref{general_lemma_omega_2_3_finite} is proved.
\end{proof}

In particular, since it is proved in \cite{BMM} that the group
$W$ can only have the orders $4$, $6$, $8$ and $\infty$, Lemma
\ref{lemma_omega_2>=<3} and Proposition \ref{general_lemma_omega_2_3_finite}
immediately lead  to the following result.
     \begin{prop}
     \label{lemma_omega_2_3_finite}
          The walk is associated with a group $W=\langle \Psi,\Phi\rangle$ of order $4$ if and only if for all $z\in]0,1/k[$,
          $\omega_{3}=\omega_{2}/2$. For $k\in\{3,4\}$, the walk has a group of order $2k$ and a negative
          (resp.\ positive) covariance \eqref{def_covariance} if and only if for all $z\in]0,1/k[$,
          $\omega_{3}=\omega_{2}/k$ (resp.\ $\omega_{3}=\omega_{2}-\omega_{2}/k)$).
     \end{prop}

The location of the reciprocal images of
$[x_{1},x_{2}]$, $[x_{3},x_{4}]$, $[y_{1},y_{2}]$ and $[y_{3},y_{4}]$
through \eqref{uniformization}
being known, see Proposition \ref{transformation_cycles_uniformization}, 
we give in Figure \ref{Ex_Pa} two examples of the
parallelogram $[0,\omega_{2}[\times $ $[0,\omega_{1}/\imath[$
with its important cycles---in
addition to the one corresponding to Gessel's walk, that can be found in \cite[Figure 5]{KRG}.

\begin{figure}[t]
\begin{center}
\begin{picture}(400.00,85.00)
\includegraphics{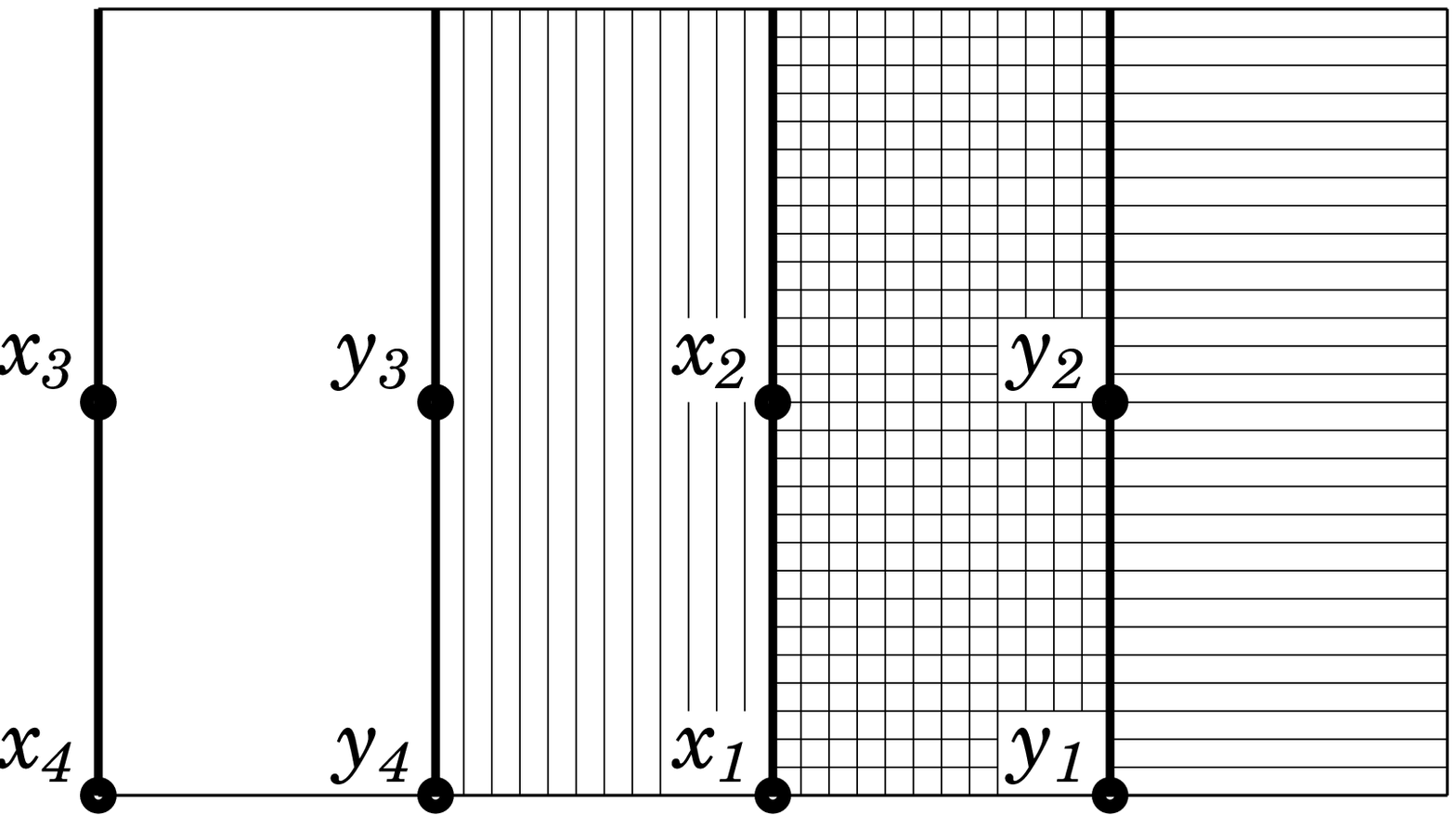}
\hspace{60mm}
\includegraphics{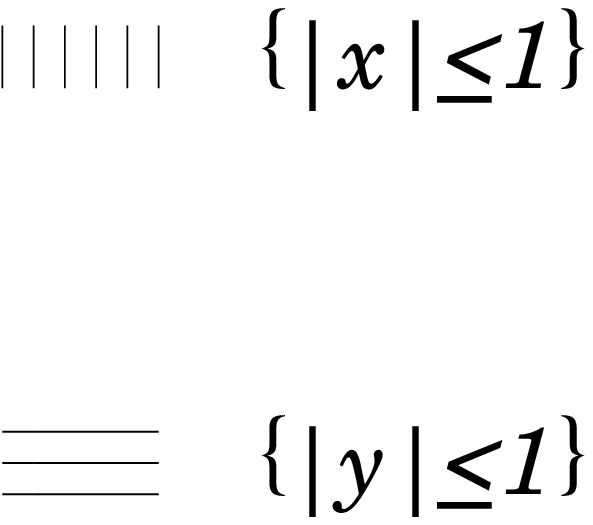}
\hspace{25mm}
\includegraphics{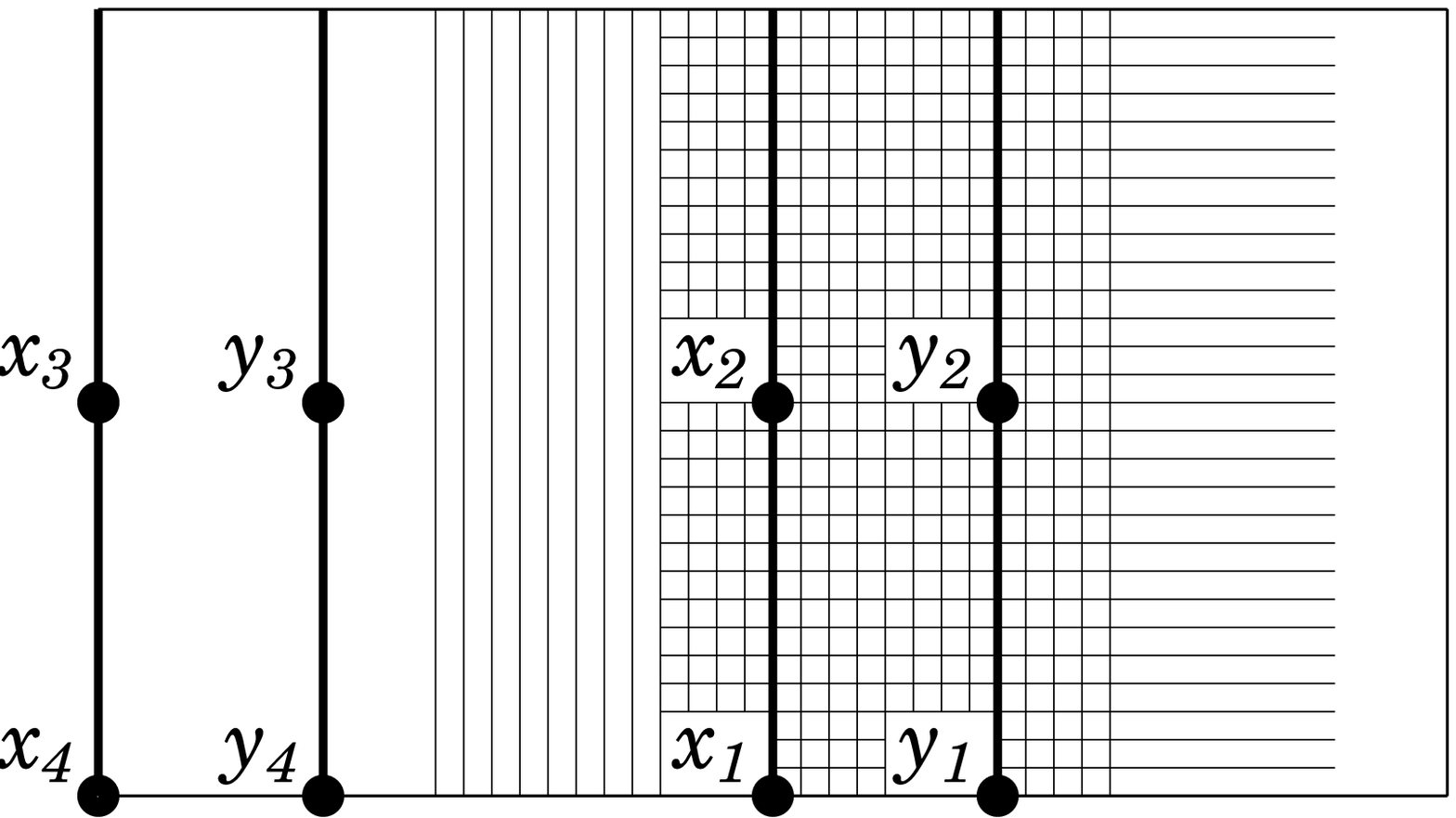}
\end{picture}
\caption{Examples of the parallelogram $[0,\omega_{2}[\times [0,\omega_{1}/\imath[$
         with its important cycles, on the left for $\{{\sf N},{\sf E},{\sf S},{\sf W}\}$
         and on the right for $\{{\sf N},{\sf E},{\sf SE},{\sf S},{\sf W},{\sf NW}\}$}
         \label{Ex_Pa}
\end{center}
\end{figure}

For the walk  $\{{\sf N},{\sf E},{\sf S},{\sf W}\}$, the group has indeed an order equal to $4$,
and Proposition \ref{lemma_omega_2_3_finite} entails $\omega_{3}
=\omega_{2}/2$. Moreover, it is easy to show that in this case $X([y_{1},
y_{2}])=X([y_{3},y_{4}])=$ $Y([x_{1},x_{2}])=Y([x_{3},x_{4}])$ coincides with the unit
circle. In particular, we have $x^{-1}(\{|x|=1\})=([0,\omega_{1}[+\omega_{2}/4)
\cup ([0,\omega_{1}[+3\omega_{2}/4)$ and $y^{-1}(\{|y|=1\})=([0,\omega_{1}[)\cup([0,
\omega_{1}[+\omega_{2}/2)$. See on the left in Figure \ref{Ex_Pa}.

Further, for the walk $\{{\sf N},{\sf E},{\sf SE},{\sf S},{\sf W},{\sf NW}\}$, the group has order $6$ and
the covariance is negative, hence by Proposition \ref{lemma_omega_2_3_finite} we
obtain $\omega_{3}=\omega_{2}/3$. Then by the same arguments as in the proof
of \cite[Proposition 26]{KRG}, we obtain the location of the cycles $x^{-1}(\{|x|=1\})$
and $y^{-1}(\{|y|=1\}$ for this walk, namely, $x^{-1}(\{|x|=1\})=([0,\omega_{1}[+
\omega_{2}/4)\cup([0,\omega_{1}[+3\omega_{2}/4)$ and $y^{-1}(\{|y|=1\}=([0,
\omega_{1}[+5\omega_{2}/12)\cup ([0,\omega_{1}[+11\omega_{2}/12)$.
See on the right in Figure \ref{Ex_Pa}.

\section{Conformal gluing functions}
\label{Study_CGF}

The main subject of Section \ref{Study_CGF} is to introduce and to study suitable CGFs
(see Definition \ref{def_CGF}) $w$ and $\widetilde{w}$ for the sets $\mathscr{G}X([y_{1},y_{2}])$ and $\mathscr{G}
Y([x_{1},x_{2}])$ in all the $74$ non-degenerate cases. For this purpose, we first find explicitly,
in Subsection \ref{fapC}, all appropriate CGFs for
both sets above (essentially thanks to the uniformization studied in Section \ref{Uniformization}),
and in particular we prove Theorem \ref{ann}. Then we observe that the behavior of these
CGFs is strongly influenced by the finiteness of the group $W$ defined in \eqref{def_generators_group}. 
Accordingly, we study
separately, in Subsections \ref{CGF_infinite} and \ref{CGF_finite}, the walks having an infinite
and then a finite group, and we show Theorems \ref{nature_CGF} and \ref{big_theo}.

\subsection{Finding all suitable conformal gluing functions}
\label{fapC}

The book \cite{FIM} provides explicitly \emph{one} CGF, and it does so \emph{only} for $z=1/k$.
We begin here by generalizing this result and by finding
the expressions of \emph{all} possible CGFs for the sets $X([y_{1},y_{2}])$ and $Y([x_{1},x_{2}])$
for \emph{any} value of $z$. Start by quoting \cite[page 126]{FIM}: if we note
$\widehat{w}=w\circ x$ or $\widehat{w}=w\circ y$ with $x,y$ defined in \eqref{uniformization}, then
the problem of finding a CGF $w$ is equivalent to \emph{finding a function $\widehat{w}$ meromorphic
in $[\omega_{2}/2,(\omega_{2}+\omega_{3})/2]\times \mathbb{R}$, $\omega_{1}$-periodic, with
only one simple pole in the domain $[\omega_{2}/2,(\omega_{2}+\omega_{3})/2]\times [0,\omega_{1}/\imath ]$
(this domain is hatched on the left in Figure \ref{Two_Par}), and satisfying to the next two conditions:}
\begin{enumerate}
     \item \label{fi} \emph{For all} $\omega\in[-\omega_{1}/2,\omega_{1}/2]$, $\widehat{w}
           ([\omega_{1}+\omega_{2}]/2+\omega)=\widehat{w}([\omega_{1}+\omega_{2}]/2-\omega)$;
     \item \label{tw} \emph{For all} $\omega\in[-\omega_{1}/2,\omega_{1}/2]$, $\widehat{w}
           ([\omega_{1}+\omega_{2}+\omega_{3}]/2+\omega)=\widehat{w}([\omega_{1}+\omega_{2}
           +\omega_{3}]/2-\omega)$.
\end{enumerate}

Setting $\overline{w}(\omega)=\widehat{w}([\omega_{1}+\omega_{2}]/2+\omega)$, the problem mentioned
above becomes that of \emph{finding a function $\overline{w}$ meromorphic in
$[0,\omega_{3}/2]\times \mathbb{R}$, $\omega_{1}$-periodic, with only one simple pole in the
domain $[0,\omega_{3}/2]\times [-\omega_{1}/(2\imath),\omega_{1}/(2\imath)]$, and such that:}
\begin{enumerate}[label=(\roman{*}'),ref=(\roman{*}')]
     \item \label{fip} \emph{For all} $\omega\in[-\omega_{1}/2,\omega_{1}/2]$,
           $\overline{w}(\omega)=\overline{w}(-\omega)$;
     \item \label{twp} \emph{For all} $\omega\in[-\omega_{1}/2,\omega_{1}/2]$,
           $\overline{w}(\omega_{3}/2+\omega)=\overline{w}(\omega_{3}/2-\omega)$.
\end{enumerate}

Now we notice that by analytic continuation, \ref{fip} allows us to continue the function $\overline{w}$ from
$[0,\omega_{3}/2]\times \mathbb{R}$ up to $[-\omega_{3}/2,\omega_{3}/2]\times \mathbb{R}$, and
next, also by analytic continuation, \ref{twp} enables us to continue $\overline{w}$ as a
$\omega_{3}$-periodic function---since evaluating \ref{twp} at $\omega_{3}/2+\omega$ and using \ref{fip} lead
to $\overline{w}(\omega_{3}+\omega)=\overline{w}(\omega)$. In particular, the problem of finding the CGF
finally becomes the following: \emph{find $\overline{w}$ an even elliptic function with periods $\omega_{1},
\omega_{3}$ with only two simple poles at $\pm p$ (or one double pole at $p$ if $p$ and $-p$ are
congruent modulo the lattice) in the parallelogram $[-\omega_{3}/2,\omega_{3}/2]\times [-\omega_{1}
/(2\imath),\omega_{1}/(2\imath)]$,} see on the right in Figure \ref{Two_Par}.

\medskip

\begin{figure}[t]
\begin{center}
\begin{picture}(400.00,90.00)
\includegraphics{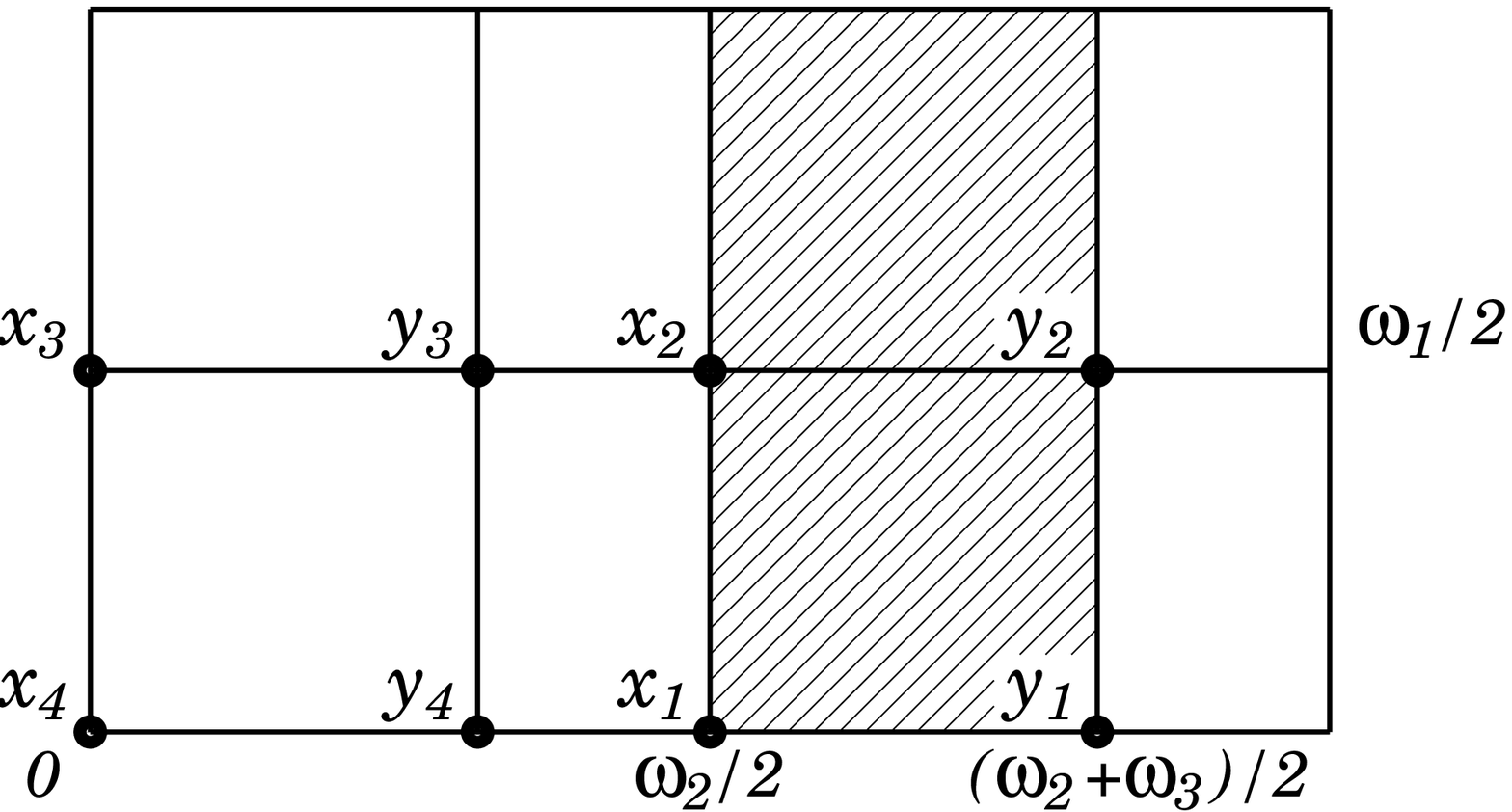}
\hspace{80mm}
\includegraphics{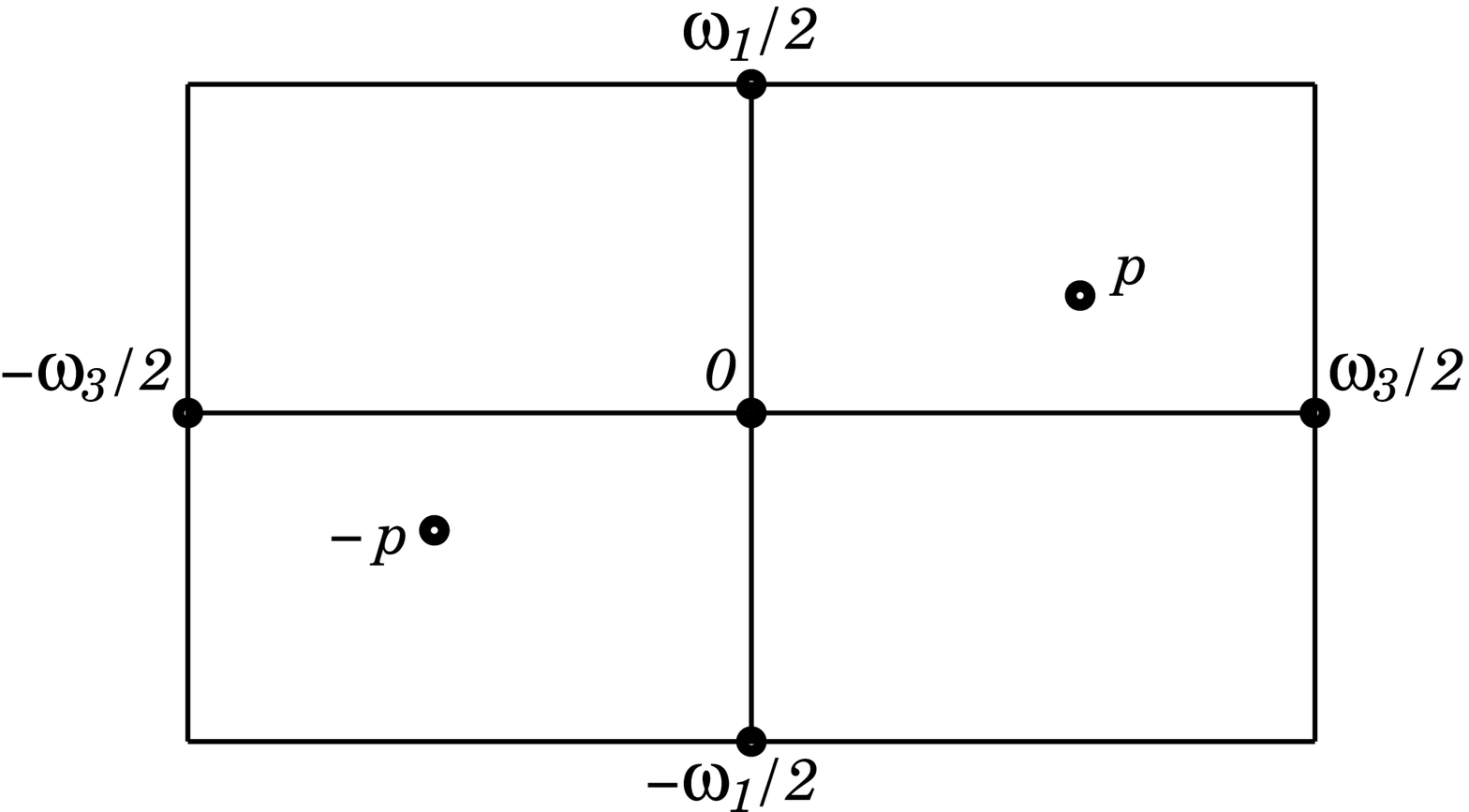}
\end{picture}
\end{center}
\caption{Domains of definition of $\widehat{w}$ (on the left) and $\overline{w}$ (on the right)}
\label{Two_Par}
\end{figure}

Let $\wp_{1,3}$ denote the Weierstrass elliptic function with periods $\omega_{1},\omega_{3}$.
A crucial fact for us is the following.

\begin{lem}
\label{lemma_properties_Weierstrass}
Let $p\in[-\omega_{3}/2,\omega_{3}/2]\times [-\omega_{1}/(2\imath),\omega_{1}/(2\imath)]$.
\begin{itemize}
\item[---] If $p=0$, the solutions of the problem above are $\{\alpha+
\beta \wp_{1,3}(\omega): \alpha,\beta\in\mathbb{C}\}$.
\item[---] If $p\neq 0$, the solutions are $\{\alpha+\beta/[\wp_{1,3}(\omega)-\wp_{1,3}(p)]:
\alpha,\beta\in\mathbb{C}\}$.
\end{itemize}
\end{lem}
\noindent{\textit{Proof.}}
It is a well-known fact, see, \textit{e.g.}, \cite[Theorem 3.11.1 and Theorem 3.13.1]{JS}, that an
even elliptic function with periods $\omega_{1},\omega_{3}$ having $2q$ poles in a
period parallelogram is necessarily a rational function of order $q$ of $\wp_{1,3}$.
In our case, $\overline{w}$, which has exactly two poles of order one or one pole of order two
in $[-\omega_{3}/2,\omega_{3}/2]\times [-\omega_{1}/(2\imath),\omega_{1}/(2\imath)]$, is thus
a fractional linear transformation of $\wp_{1,3}$.
\begin{itemize}
\item[---] In particular, it is immediate that $p=0$ yields $\overline{w}(\omega)=\alpha + \beta \wp_{1,3}(\omega)$.
\item[---] If $p\neq 0$ then $\wp_{1,3}(p)\neq \infty$, and we get $\overline{w}(\omega)
=[\alpha \wp_{1,3}(\omega)+\gamma]/[\wp_{1,3}(\omega)-\wp_{1,3}(p)]$.\hfill $\square$
\end{itemize}

Since by \eqref{uniformization} we have $x(\omega)=f^{-1}(\wp(\omega))$ with
$f$ defined in \eqref{def_f},
applying Lemma \ref{lemma_properties_Weierstrass} for $p=0$ implies the following theorem.


\begin{thm}
\label{ann}The functions
     \begin{equation}
     \label{expression_CGFx}
          w(t)=\wp_{1,3}(\wp^{-1}(f(t))-[\omega_{1}+\omega_{2}]/2),
          \ \ \ \ \ \widetilde{w}(t)=w(X_0(t)),
     \end{equation}
are suitable CGFs for the sets $X([y_{1},y_{2}])$ and $Y([x_{1},x_{2}])$, respectively.
\end{thm}

\begin{rem}
{\rm By \eqref{uniformization} we have $x(\omega)=X_{0}(y(\omega))$
so that, of course,
$\widetilde{w}(t)=\wp_{1,3}(y^{-1}(t)-[\omega_{1}+\omega_{2}]/2)$.}
\end{rem}

The CGF $w$ is therefore equal to the Weierstrass elliptic function with
periods $\omega_{1},\omega_{3}$ evaluated at a translation of the reciprocal
of the Weierstrass elliptic function with periods $\omega_{1},\omega_{2}$.
It turns out that the theory of transformation of elliptic functions---the
basic result of which being here recalled in \eqref{principle_transformation}
below---entails that this expression admits a wonderful simplification if
$\omega_{3}/\omega_{2}$ is rational.

But Proposition \ref{general_lemma_omega_2_3_finite} shows that the latter condition is
related to the group $W=\langle \Psi,\Phi\rangle$ defined in \eqref{def_generators_group},
since it states that $\omega_{3}/\omega_{2}$ is rational for all $z\in]0,1/k[$ if and
only if the group $W$ is finite. Consequently, we consider separately, in 
Subsections \ref{CGF_infinite} and \ref{CGF_finite} below, the study of $w$ according to the
finiteness of this group. It is worth noting that, to the best of our knowledge, these 
upcoming results on an in-depth study of the CGFs are new, even for $z=1/k$.

\begin{rem}
{\rm Given a CGF $u$ for a generic set $\mathscr{C}$, any of its fractional
linear transformations $[\alpha u+\beta]/[\gamma u+\delta]$ with $\alpha,\beta,\gamma,
\delta \in \mathbb{C}$ such that $\alpha,\gamma\neq 0$ is also a CGF for $\mathscr{C}$,
see Definition \ref{def_CGF}\footnote{Incidentally, we note that the quantity
$\partial_t u(t)/[u(t)-u(x)]$ appearing in Theorem \ref{explicit_integral} is
invariant through the transformation $u\mapsto [\alpha u+\beta]/[\gamma u+\delta]$,
as soon as $\alpha,\gamma\neq 0$.}. They are actually the only ones, see \cite[Chapter 2]{LIT}
for a proof using Fredholm operators.
Applying Lemma \ref{lemma_properties_Weierstrass} to any value
of $p$ enables us to recover this in an elementary way. This has also two interesting
consequences. First, this means that Theorem \ref{ann} gives the explicit expression
of \emph{all} possible CGFs. Also, this implies that Theorem \ref{nature_CGF}
can be extended to \emph{any} CGFs; in other words, the nature of the CGFs is
intrinsic, in the sense that it only depends on the kernel.}
\end{rem}

\begin{rem}
{\rm We recall from \cite[Section 4]{KR} the following global
properties of the CGFs.
The function $w$ (resp.\ $\widetilde{w}$)
defined by \eqref{expression_CGFx} is meromorphic on $\mathbb{C}\setminus [x_{3},x_{4}]$
(resp.\ $\mathbb{C}\setminus [y_{3},y_{4}]$). It has therein one simple pole, at $x_{2}$
(resp.\ $Y(x_{2})$), and $\lfloor\omega_{2}/(2\omega_{3})\rfloor$\footnote{For $r\in\mathbb{R}$, $\lfloor r\rfloor$ denotes the
lower integer part of $r$, \textit{i.e.}, the unique $p\in\mathbb{Z}$
such that $p\leq r<p+1$.} double
poles at some points lying on the segment $]x_{2},x_{3}[\cap (\mathbb{C}\setminus \mathscr{G}
X([y_{1},y_{2}]))$ (resp. $]y_{2},y_{3}[\cap (\mathbb{C}\setminus \mathscr{G}Y([x_{1},
x_{2}]))$).}
\end{rem}

Before concluding this part, we state the following properties of the $\wp$-Weierstrass function, which
will be of the highest significance for the proof of Theorem \ref{big_theo}.

\begin{lem}
The following results on the $\wp$-Weierstrass elliptic function hold:
\begin{itemize}
\item[---] Let ${\wp}$ be a Weierstrass elliptic function with certain periods
$\overline{\omega},\widehat{\omega}$, and let $p$ be some positive integer.
The Weierstrass elliptic function with periods $\overline{\omega},\widehat{\omega}
/p$ can be written in terms of ${\wp}$ as
     \begin{equation}
     \label{principle_transformation}
          {\wp}(\omega)+\sum_{k=1}^{p-1}[{\wp}(\omega+
          k\widehat{\omega}/p)-{\wp}(k\widehat{\omega}/p)].
     \end{equation}
\item[---] Let ${\wp}$ be a Weierstrass elliptic function. We have
           the addition theorem: 
     \begin{equation}
     \label{addition_theorem}
          \forall \omega,\widetilde{\omega},
          \ \ \ \ \ \wp({\omega}+\widetilde{\omega})=-\wp({\omega})
          -\wp(\widetilde{\omega})+\frac{1}{4}\left[\frac{\wp'({\omega})-
          \wp'(\widetilde{\omega})}{\wp({\omega})-
          \wp(\widetilde{\omega})}\right]^{2}.
     \end{equation}
\item[---] For any integer $r\geq 0$, there exists a polynomial $p_{r}(x)$,
with dominant term equal to $(2r+1)!\,x^{r+1}$, such that
     \begin{equation}
     \label{diff_n_wp}
          \wp^{(2r)}=p_{r}(\wp).
     \end{equation}
\end{itemize}
\end{lem}

\begin{proof}
Equality \eqref{principle_transformation} is shown, \textit{e.g.}, in \cite[page 456]{WW}.
Identity \eqref{addition_theorem} is most classical, and can be found in \cite{JS,WW}.
As for Equation \eqref{diff_n_wp}, differentiating \eqref{diff_eq_wp} gives
$\wp''=6\wp^{2}-g_2/2$ (with obvious notations), and higher derivatives of even order are clearly expressible
as in \eqref{diff_n_wp}.
\end{proof}

\subsection{Case of an infinite group}
\label{CGF_infinite}

In this subsection we concentrate on the case of an \emph{infinite}
group of the walk, and we show the part of Theorem \ref{nature_CGF}
concerning the infinite group case. Precisely, we prove that for any
$z\in]0,1/k[$ such that $\omega_{3}/\omega_{2}$ is irrational, then the
CGFs $w$ and $\widetilde{w}$ defined in \eqref{expression_CGFx} are non-holonomic. 
Before starting the proof, we notice that for a
walk admitting an infinite group, it may happen that
$\omega_{3}/\omega_{2}$ is rational for certain values of $z$, but
$\omega_{3}/\omega_{2}$ also has to take irrational values. Indeed,
if $\omega_{3}/\omega_{2}$ was rational for all $z\in]0,1/k[$, then
$\omega_{3}/\omega_{2}$ would be a rational constant, since with
\eqref{def_omega_1_2} and \eqref{def_omega_3}, the quantity
$\omega_{3}/\omega_{2}$ is holomorphic of $z$, and
Proposition \ref{general_lemma_omega_2_3_finite} would then entail
that the group is finite.

\begin{proof}
Let
     \begin{equation*}
          v(t)=w(f^{-1}(t))=\wp_{1,3}(\wp^{-1}(t)-[\omega_{1}+\omega_{2}]/2).
     \end{equation*}
The class of holonomic functions being closed under algebraic substitutions,
see, \textit{e.g.}, \cite{FLAJ}, it is enough to prove that $v$
is non-holonomic. Indeed, on the one hand $f$ is rational, and on the
other hand $\widetilde{w}=w(X_{0})$, where $X_{0}$ is algebraic. We are going
to show first that $v$ is non-algebraic, and then we will prove that if  $v$ is
holonomic then it has to be algebraic, in such a way that $v$ will be non-holonomic.

\medskip

Suppose thus that $v$ is algebraic. In other words, we assume that there exist polynomials
$a_{0},\ldots ,a_{q}$ with $a_{q}\neq 0$, such that $\sum_{k=0}^{q}a_{k}(t) v(t)^{k}=0$.
By definition of $v$ and since the Weierstrass elliptic function is non-algebraic, at
least one of $a_{0},\ldots ,a_{q}$ is non-constant.
Evaluating the last equality at $t=\wp(\omega+[\omega_{1}+\omega_{2}]/2)$ and
using the definition of $v$, we get $\sum_{k=0}^{q}a_{k}(\wp(\omega+[\omega_{1}+
\omega_{2}]/2))\wp_{1,3}(\omega)^{k}=0$. Since $\wp(\omega+[\omega_{1}+\omega_{2}]
/2)$ is a rational transformation of $\wp(\omega)$, see  the addition
theorem \eqref{addition_theorem}, the previous equality yields the identity $P(\wp(\omega),
\wp_{1,3}(\omega))=0$, where $P$ is a certain polynomial, which is non-constant with
respect to the two variables.

Now we recall \cite{JS,WW}  that $\wp$ (resp.\ $\wp_{1,3}$) has poles at every point
of the lattice $\omega_{1}\mathbb{Z}+\omega_{2}\mathbb{Z}$ (resp.\ $\omega_{1}\mathbb{Z}
+\omega_{3}\mathbb{Z}$). Moreover, it is well known that $\omega_{2}\mathbb{Z}+\omega_{3}\mathbb{Z}$ 
is dense in $\mathbb{R}$ if $\omega_{3}/\omega_{2}$ is
irrational. In
particular, if $\omega_{3}/\omega_{2}$ is irrational then the poles of $P(\wp(\omega),
\wp_{1,3}(\omega))$ are not isolated, which contradicts the principle
of analytic continuation.

\medskip

Suppose now that $v$ is holonomic, \textit{i.e.}, that there exist polynomials
$a_{0},\ldots ,a_{q}$ with $a_{q}\neq 0$ such that $\sum_{k=0}^{q}a_{k}(t)v^{(k)}
(t)=0$. We show that, in this case,
     \begin{equation}
     \label{gen_met}
          \sum_{k=0}^{q}a_{k}(t) v^{(k)}(t)=U(t,v(t))+V(t,v(t))v'(t),
     \end{equation}
where the dependence of $U,V$ with respect to the first (resp.\ second) variable is
rational (resp.\ polynomial), and where at least one of $U,V$ is non-zero. These facts will
entail the algebraicity of $v$. Indeed:
\begin{itemize}
\item[---] If $V$ is identically zero, then $U$ has to be non-zero.
Moreover, since $U$ is rational with respect to the first variable and since with \eqref{gen_met} we
have $U(t,v(t))=0$, $U$ has to be non-constant with respect to the second variable,
and \eqref{gen_met} immediately yields that $v$ is then algebraic.
\item[---] Suppose now that $V$ is not identically zero. From \eqref{gen_met} and the
supposed holonomicity of $v$ it follows that $U(t,v(t))^{2}-V(t,v(t))^{2}v'(t)^{2}=0$.
Also, noting $g$ the derivative of ${\wp^{-1}}$ (if $g_{2}$ and $g_{3}$
are the invariants\footnote{They are defined by $g_2=-4[\wp(\omega_1/2)\wp([\omega_1+\omega_2]/2)+\wp(\omega_1/2)\wp(\omega_2/2)+\wp([\omega_1+\omega_2]/2)\wp(\omega_2/2)]$ and $g_3=4\wp(\omega_1/2)\wp([\omega_1+\omega_2]/2)\wp(\omega_2/2)$.} of $\wp$, we have $g(t)=1/[4t^{3}-g_{2}t-g_{3}]^{1/2}$), we
obtain $v'(t)^{2}=g(t)^{2}\wp_{1,3}'(\wp^{-1}(t)-[\omega_{1}+\omega_{2}]/2)^{2}$. \vspace{-0.4mm}
But ${\wp_{1,3}'^{2}}=4\wp_{1,3}^{3}-g_{2,1,3}\wp_{1,3}-g_{3,1,3}$,
where $g_{2,1,3}$ and $g_{3,1,3}$ are the invariants of $\wp_{1,3}$.
Finally, we have that
$U(t,v(t))^{2}-V(t,v(t))^{2}g(t)^{2}[4v(t)^{3}-g_{2,1,3}v(t)-g_{3,1,3}]=0$. This
last quantity is a rational function of the first variable, since $g(t)^{2}$ is rational,
and an odd-degree polynomial in the second one, in such a way that $v$ is algebraic.
\end{itemize}
So, it is definitely enough to prove \eqref{gen_met}. For this we are going to show that for any $k\geq 0$,
     \begin{equation}
     \label{par_met}
          v^{(k)}(t)=U_{k}(t,v(t))+V_{k}(t,v(t))v'(t),
     \end{equation}
where the dependence of $U_{k}$ (resp.\ $V_{k}$) is rational with respect to the first variable and
polynomial of degree exactly $\lfloor k/2+1\rfloor$ (resp.\ $\lfloor (k-1)/2\rfloor$)
with respect to the second one.
Equality \eqref{gen_met} will then be an immediate consequence of \eqref{par_met}.
Indeed, if $q$ is even then by using \eqref{par_met} in \eqref{gen_met} we obtain
that the degree of $U$ in $v$ is exactly $\lfloor q/2+1\rfloor$ and $U$ is thus
obviously non-zero. Likewise, if $q$ is odd we get that the degree of $V$ is
exactly $\lfloor (q-1)/2\rfloor$ and is thus clearly non-zero if $q\geq 3$.
If $q=1$ then we easily make explicit $V_{0}$ ($=0$) and $V_{1}$, and we immediately
deduce that $V$ is also non-zero.

We now show \eqref{par_met}. For $k=0$, this is obvious. For $k\geq 1$,
a straightforward calculation leads to
     \begin{equation*}
          v^{(k)}(t)=\sum_{p=1}^{k}b_{p}(t)\wp_{1,3}^{(p)}(\wp^{-1}(t)-[\omega_{1}+\omega_{2}]/2),
     \end{equation*}
with $b_{k}=g(t)^{k}$, $b_{k-1}(t)=[k(k-1)/2]g'(t)g(t)^{k-2}$, and for $p\in\{2,\ldots ,k-1\}$,
$b_{k-p}$ is a polynomial in the variables $g(t),g'(t),\ldots ,g^{(p)}(t)$. Moreover, if $p$
is even then $b_{p}$ clearly is rational, whereas if $p$ is odd then $b_{p}/g$ is rational. Next,
by a repeated use of \eqref{diff_n_wp} and by using that $\wp_{1,3}'(\wp^{-1}(t)-[
\omega_{1}+\omega_{2}]/2)=v'(t)/g(t)$, we obtain
     \begin{equation*}
          v^{(k)}(t)=\sum_{\substack{p=1 \\p\ \text{even}}}^k b_{p}(t)p_{p/2}(v(t))+
          v'(t)\sum_{\substack{p=1 \\p\ \text{odd}}}^k [b_{p}(t)/g(t)]p_{(p-1)/2}'(v(t)).
     \end{equation*}
With the values of $b_{k}$ and $b_{k-1}$ given above this immediately yields
\eqref{par_met}, and therefore also \eqref{gen_met}, and finally the fact that
$v$, $w$ and $\widetilde{w}$ are non-holonomic.
\end{proof}

\subsection{Case of a finite group}
\label{CGF_finite}

In this subsection, we consider the case of a \emph{finite} group. We
show that $w$ and $\widetilde{w}$ are then algebraic, and we complete the proof
of Theorem \ref{nature_CGF}. Moreover, we  considerably simplify the explicit
expressions of $w$ and $\widetilde{w}$ given in \eqref{expression_CGFx} for 22 of the 23 walks
having a finite group\footnote{The 23rd model---namely, Gessel's walk---has
already been treated in \cite{KRG}.}: see Theorem \ref{big_theo} \ref{prop_CGF_order_4} 
(resp. \ref{prop_CGF_order_6_-} and \ref{prop_CGF_order_6_+},
\ref{prop_CGF_order_8}) for the walks admitting a group of order $4$ (resp.\ $6$, $8$).
 Thanks to Proposition \ref{general_lemma_omega_2_3_finite}, we obtain that for the $23$ walks having a finite group,
the quantity $\omega_{2}/\omega_{3}$ is rational. Moreover, with \cite{BMM} and
Proposition \ref{lemma_omega_2_3_finite}, it is sufficient to consider the cases where
$\omega_{2}/\omega_{3}$ are equal to $2,3,3/2,4,4/3$. According to the classification
of \cite{BMM}---recalled here in Figure \ref{Tables}---there are $16,2,3,1,1$ such walks, respectively.


\paragraph{Finite group of the walk and negative or zero covariance.}
We first consider the case of a
covariance equal to zero---or equivalently the case $\omega_{2}/\omega_{3}=2$.

\begin{proof}[Proof of Theorem {\rm\ref{big_theo} \ref{prop_CGF_order_4}}.]
Actually, it suffices to show
the following identity (specific to the case $\omega_{2}/\omega_{3}=2$):
     \begin{equation}
     \label{func_eq_wp_4}
          \wp_{1,3}(\omega-[\omega_{1}+\omega_{2}]/2)=
          -2f(x_{1})+[f(x_{2})-f(x_{3})]^{2}
          \frac{\wp(\omega)-f(x_{1})}{[\wp(\omega)-f(x_{2})]
          [\wp(\omega)-f(x_{3})]}.
     \end{equation}
Indeed, if Equation \eqref{func_eq_wp_4} holds, let us evaluate it at $\omega=\wp^{-1}(f(t))$:
with \eqref{expression_CGFx} we obtain that there exist two constants $K_{1}$
and $K_{2}$ such that $w(t)=K_{1}+K_{2}[f(t)-f(x_{1})]/[(f(t)-f(x_{2}))(f(t)-
f(x_{3}))]$. But by using the explicit expression \eqref{def_f} of $f$, it is immediate that
if $x_{4}\neq \infty$ then $w(t)=K_{1}+K_{3}[(t-x_{1})(t-x_{4})]/[(t-x_{2})(t-
x_{3})]$, and if $x_{4}=\infty$ then $w(t)=K_{1}+K_{4}[t-x_{1}]/[(t-x_{2})(t-x_{
3})]$, $K_{3}$ and $K_{4}$ being some non-zero constants.

To prove \eqref{func_eq_wp_4}, start by applying \eqref{principle_transformation} for $p=2$: we obtain
$\wp_{1,3}(\omega-[\omega_{1}+\omega_{2}]/2)=\wp(\omega-[\omega_{1}+\omega_{2}]/2)+
\wp(\omega-\omega_{1}/2)-\wp(\omega_{2}/2)$. Now, taking the usual notations $e_{1}=
\wp(\omega_{1}/2)$, $e_{1+2}=\wp([\omega_{1}+\omega_{2}]/2)$ and $e_{2}=\wp(\omega_{2}/2)$,
we can state the two following particular cases of the addition formula \eqref{addition_theorem}:
$\wp(\omega-[\omega_{1}+\omega_{2}]/2)=e_{1+2}+[(e_{1+2}-e_{1})(e_{1+2}-e_{2})]/[\wp(\omega)-e_{1+2}]$
and $\wp(\omega-\omega_{1}/2)=e_{1}+[(e_{1}-e_{2})(e_{1}-e_{1+2})]/[\wp(\omega)-e_{1}]$.
In this way and after simplification we obtain $\wp_{1,3}(\omega-[\omega_{1}+\omega_{2}]/2)=
e_{1}-e_{2}+e_{1+2}+(e_{1}-e_{1+2})^{2}[\wp(\omega)-e_{2}]/[(\wp(\omega)-e_{1})(\wp(\omega)-e_{1+2})]$.
Finally, by using the equalities $e_{1}=f(x_{3})$, $e_{1+2}=f(x_{2})$ and $e_{2}=f(x_{1})$,
see below \eqref{def_f}, as well as $e_{1}+e_{1+2}+e_{2}=0$, we immediately obtain \eqref{func_eq_wp_4}.
\end{proof}

Let us now consider the case of a negative covariance---or
equivalently $\omega_{2}/\omega_{3}$ equal to $3$ or $4$.
We start with the situation where $\omega_{2}/\omega_{3}=3$.

\begin{proof}[Proof of Theorem {\rm\ref{big_theo} \ref{prop_CGF_order_6_-}}.] First of all, let us give a sketch of the
proof. We shall first show the existence and find the expression of two third-degree polynomials
$A$ and $B$ such that $\wp_{1,3}(\omega-[\omega_{1}+\omega_{2}]/2)=A(\wp(\omega))/B(\wp(\omega))$.
In particular, with \eqref{expression_CGFx}, we shall then obtain that $w(t)=A(f(t))/B(f(t))$. Next
we shall consider the two walks $\{{\sf N},{\sf SE},{\sf W}\}$ and 
$\{{\sf N},{\sf E},{\sf SE},{\sf S},{\sf W},{\sf NW}\}$ separately.
\begin{itemize}
  \item[---]  For the walk $\{{\sf N},{\sf SE},{\sf W}\}$, we have
        $x_{4}=\infty$ and then $A(f(t))$ and $B(f(t))$ are also third-degree polynomials,
        see the expression \eqref{def_f} of $f$. We will show that $B(f(t))
        =$ $(t-x_{2})(t-1/{x_{2}}^{1/2})^{2}$. In addition, if $r(t)$ denotes the rest of the
        euclidean division of $A(f(t))$ by $B(f(t))$, we will prove that $r(t)=(r''(0)/2)
        t^{2}$ with $r''(0)\neq 0$, in such a way that $w(t)=B(f(\infty))/A(f(\infty))+
        (r''(0)/2)t^{2}/[(t-x_{2})(t-1/{x_{2}}^{1/2})^{2}]$.
  \item[---] For the walk $\{{\sf N},{\sf E},{\sf SE},{\sf S},{\sf W},{\sf NW}\}$, we have
        $x_{4}\neq \infty$ and then $(t-x_{4})^{3}A(f(t))$ and $(t-x_{4})^{3}B(f(t))$ are
        third-degree polynomials. We will show that $(t-x_{4})^{3}B(f(t))=(t-x_{2})(t-1/
        {x_{2}}^{1/2})^{2}$. If we denote by $r(t)$  the rest of the euclidean
        division of $(t-x_{4})^{3}A(f(t))$ by $(t-x_{4})^{3}B(f(t))$, then we will prove
        that $r(t)=(r''(0)/2)t(t+1)$ with $r''(0)\neq 0$ so that $w(t)=B(f(\infty))/
        A(f(\infty))+(r''(0)/2)t(t+1)/[(t-x_{2})(t-1/{x_{2}}^{1/2})^{2}]$.
\end{itemize}
Finally, with the explicit formulation of $w$ and the fact
that $\widetilde{w}=w(X_{0})$,  an elementary calculation will lead, for each walk,
to the expression of $\widetilde{w}$ stated in Theorem \ref{big_theo} \ref{prop_CGF_order_6_-};
this will conclude the proof.

\medskip

So we begin by finding explicitly, for both $\{{\sf N},{\sf SE},{\sf W}\}$ 
and $\{{\sf N},{\sf E},{\sf SE},{\sf S},{\sf W},{\sf NW}\}$, two
polynomials $A$ and $B$ of degree three such that $\wp_{1,3}(\omega-[\omega_{1}+
\omega_{2}]/2)=A(\wp(\omega))/B(\wp(\omega))$.

Applying \eqref{principle_transformation} with $p=3$, we get
$\wp_{1,3}(\omega-[\omega_{1}+\omega_{2}]/2)=\wp(\omega-[\omega_{1}+
\omega_{2}]/2)+\wp(\omega-\omega_{1}/2-\omega_{2}/6)-\wp(\omega_{2}/3)+
\wp(\omega-\omega_{1}/2+\omega_{2}/6)-\wp(2\omega_{2}/3)$. Then, using
the addition formula \eqref{addition_theorem} for $\wp$,
noting $K=e_{1+2}-2\wp(\omega_{2}/3)-2\wp(\omega_{1}/2+\omega_{2}/6)$,
using that
$\wp(\omega_{1}/2+\omega_{2}/6)=\wp(\omega_{1}/2-\omega_{2}/6)$ and
$\wp'(\omega_{1}/2+\omega_{2}/6)=-\wp'(\omega_{1}/2-\omega_{2}/6)$---got from the fact that $\wp(\omega_{1}/2+\omega)$ is even and
$\wp'(\omega_{1}/2+\omega)$ is odd---we have that $\wp_{1,3}(\omega-[\omega_{1}+\omega_{2}]/2)$ equals
     \begin{equation}
     \label{before_A_and_B}
          \frac{(e_{1+2}-e_{1})(e_{1+2}-e_{2})}
          {\wp(\omega)-e_{1+2}}-2\wp(\omega)+\frac{1}{2}\frac{\wp'(\omega)^{2}+\wp'(\omega_{1}
          /2+\omega_{2}/6)^{2}}{[\wp(\omega)-\wp(\omega_{1}/2+\omega_{2}/6)]^{2}}+K.
     \end{equation}

With Equations \eqref{diff_eq_wp} and \eqref{before_A_and_B}, it is now clear that
$\wp_{1,3}(\omega-[\omega_{1}+\omega_{2}]/2)$ can be written as
$A(\wp(\omega))/B(\wp(\omega))$; moreover we can take $B(\wp(\omega))
=(\wp(\omega)-e_{1+2})(\wp(\omega)-\wp(\omega_{1}/2+\omega_{2}/6))^{2}$.

Now we show that $\wp(\omega_{1}/2+\omega_{2}/6)=f(1/{x_{2}}^{1/2})$.
For this we will express $\wp(\omega_{1}/2+\omega_{2}/6)$ with respect to $z$.
Since $\wp(\omega_{1}/2+\omega_{2}/6)=e_{1}+[(
e_{1}-e_{2})(e_{1}-e_{1+2})]/[\wp(\omega_{2}/6)-e_{1}]$
it is enough to express $\wp(\omega_{2}/6)$ with respect to $z$. For this we will
first find $\wp(\omega_{2}/3)$ explicitly and we will then use, for $\omega=
\omega_{2}/3$, the fact $\wp(\omega/2)$ is equal to 
     \begin{equation*}
          \wp(\omega)+[(\wp(\omega)-e_{1})(\wp(\omega)-e_{2})]^{1/2}
          \hspace{-0.3mm}+[(\wp(\omega)-e_{1})(\wp(\omega)-e_{1+2})]^{1/2}
          \hspace{-0.3mm}+[(\wp(\omega)-e_{2})(\wp(\omega)-e_{1+2})]^{1/2},
     \end{equation*}
see \cite{JS,WW}.
In other words, for all coefficients in \eqref{before_A_and_B}
to be explicit with respect to $z$, it is enough to find
only $\wp(\omega_{2}/3)$ in terms of $z$.

And now we show that for both $\{{\sf N},{\sf SE},{\sf W}\}$ 
and $\{{\sf N},{\sf E},{\sf SE},{\sf S},{\sf W},{\sf NW}\}$, $\wp(\omega_{2}/3)=1/3$. For this
we use the following fact, already recalled in \cite{KRG}: the quantity 
$x=\wp(\omega_{2}/3)$ is the only positive root of 
\begin{equation*}
x^{4}-g_{2}x^{2}/2-g_{3}x-g_{2}^{2}/48, 
\end{equation*}
$g_{2}=-4[e_{1}e_{2}+e_{1}e_{1+2}
+e_{2}e_{1+2}]$ and $g_{3}=4e_{1}e_{2}e_{1+2}$ being the invariants of $\wp$. By using the explicit
expressions of $e_{1}$, $e_{2}$ and $e_{1+2}$---see the proof of Theorem \ref{big_theo} \ref{prop_CGF_order_4}---we easily show that
$1/3$ is a root of the polynomial above; $1/3$ being positive we get $\wp(\omega_{2}/3)=1/3$.

Then an elementary calculation leads to $\wp(\omega_{1}/2+\omega_{2}/6)=f(1/{x_{2}}^{1/2})$.
Next with \eqref{diff_eq_wp} we also obtain $\wp'(\omega_{1}/2+\omega_{2}/6)$ and thus all
coefficients in \eqref{before_A_and_B} are known in terms of $z$. In particular, this is
also the case for the polynomials $A$ and $B$. After a lengthy but easy calculation we
obtain the facts claimed in the items above, and thus
Theorem \ref{big_theo} \ref{prop_CGF_order_6_-}.
\end{proof}

\begin{proof}[Proof of Theorem \ref{big_theo} \ref{prop_CGF_order_8}.]
This case, concerning $\omega_{2}/\omega_{3}=4$, could
be obtained by applying \eqref{principle_transformation} for $p=4$; the details
would be essentially the same as above, so that we omit them.
\end{proof}

\paragraph{Finite group of the walk and positive covariance.}
The only such possible walks are $\{{\sf NE},{\sf S},{\sf W}\}$, $\{{\sf N},{\sf E},{\sf SW}\}$,
$\{{\sf N},{\sf NE},{\sf E},{\sf S},{\sf SW},{\sf W}\}$, as well as Gessel's walk
$\{{\sf E},{\sf SW},{\sf W},{\sf NE}\}$.
The latter has already been considered in \cite{KRG}: there
we have shown that the CGFs $w$ and $\widetilde{w}$ defined by \eqref{expression_CGFx} are
algebraic (of degree three in $t$), and we have found their minimal polynomials.

By using the same key idea as in \cite{KRG}---namely, a double application of
\eqref{principle_transformation}---we are now going to prove Theorem \ref{big_theo}
\ref{prop_CGF_order_6_+}, \textit{i.e.}, to show that for the three walks $\{{\sf NE},{\sf S},{\sf W}\}$, $\{{\sf N},{\sf E},{\sf SW}\}$
and $\{{\sf N},{\sf NE},{\sf E},{\sf S},{\sf SW},{\sf W}\}$,  $w$ and $\widetilde{w}$ are algebraic (of degree two in $t$), and
to find their minimal polynomials.

\medskip

\begin{proof}[Proof of Theorem {\rm\ref{big_theo} \ref{prop_CGF_order_6_+}}.] 
First recall that for the three walks considered here $\omega_{2}/
\omega_{3}=3/2$, and define the auxiliary period $\omega_{4}=\omega_{2}/3$.

First, with $\omega_{4}=\omega_{2}/3$ and \eqref{principle_transformation},
we shall be able to express $\wp_{1,4}$ as a
rational function of $\wp$. Moreover, since $\omega_{4}=\omega_{3}/2$ and once again with
\eqref{principle_transformation}, we shall write $\wp_{1,4}$ as a rational function of $\wp_{1,3}$.
As an immediate consequence, $\wp_{1,3}$ will be an algebraic function of $\wp$. Then, with
\eqref{expression_CGFx} and the addition formula \eqref{addition_theorem}, we shall obtain that
the CGF $w$ defines an algebraic function of $t$.

\medskip

{\noindent \emph{Rational expression of $\wp_{1,4}$ in terms of $\wp$}}.
By using exactly the same arguments as in the proof of Theorem \ref{big_theo} \ref{prop_CGF_order_6_-},
we obtain the following three facts: firstly $\wp_{1,4}(\omega-[\omega_{1}+\omega_{2}]/2)$ is
equal to \eqref{before_A_and_B}; secondly $\wp(\omega_{2}/3)=1/3$; and thirdly the expressions of
all coefficients in \eqref{before_A_and_B} with respect to $z$ are explicit. This way, we conclude that
there exist $K_{1}$ and $K_{2}$ which depend only on $z$---and could be made explicit---such that
     \begin{equation}
     \label{wp_14_wp}
          \wp_{1,4}(x^{-1}(t)-[\omega_{1}+\omega_{2}]/2)=
          K_{1}+\frac{K_{2}u(t)}{[t-x_{2}][t-1/{x_{2}}^{1/2}]^{2}},
     \end{equation}
with $u(t)$ as described in the statement of Theorem \ref{big_theo} \ref{prop_CGF_order_6_+}.

\medskip

{\noindent \emph{Rational expression of $\wp_{1,4}$ in terms of $\wp_{1,3}$}}.
Applying, as in the proof of Theorem \ref{big_theo} \ref{prop_CGF_order_4},
the identity \eqref{principle_transformation} for $p=2$, we obtain that $\wp_{1,4}(\omega)
=\wp_{1,3}(\omega)+\wp_{1,3}(\omega+\omega_{3}/2)-\wp_{1,3}(\omega_{3}/2)$.
Noting then $e_{1,1,3}=\wp_{1,3}(\omega_{1}/2)$, $e_{1+3,1,3}=\wp_{1,3}(
[\omega_{1}+\omega_{3}]/2)$ and $e_{3,1,3}=\wp_{1,3}(\omega_{3}/2)$ we have
$\wp_{1,3}(\omega+\omega_{3}/2)=e_{3,1,3}+[(e_{3,1,3}-e_{1,1,3})(e_{3,1,3}-
e_{1+3,1,3})]/[\wp_{1,3}(\omega)-e_{3,1,3}]$. In particular, we immediately
obtain that
     \begin{equation}
     \label{wp_13_wp_14}
          \wp_{1,3}(\omega)^{2}-[e_{3,1,3}+\wp_{1,4}(\omega)]
          \wp_{1,3}(\omega)+[(e_{3,1,3}-e_{1,1,3})(e_{3,1,3}-
          e_{1+3,1,3})+e_{3,1,3}\wp_{1,4}(\omega)]=0.
     \end{equation}
Therefore, once the expressions of $e_{1,1,3}$, $e_{1+3,1,3}$ and $e_{3,1,3}$
will be known explicitly, Equations \eqref{expression_CGFx}, \eqref{wp_14_wp} and
\eqref{wp_13_wp_14} will immediately entail Theorem \ref{big_theo} \ref{prop_CGF_order_6_+}.

\medskip

It thus remains for us to find explicitly $e_{1,1,3}$, $e_{1+3,1,3}$ and $e_{3,1,3}$.
This will be a consequence of the possibility of expanding $\wp_{1,4}$ in two
different ways.

First, we saw above that $\wp_{1,4}(\omega)=\wp_{1,3}(\omega)+[(e_{3,1,3}-
e_{1,1,3})(e_{3,1,3}-e_{1+3,1,3})]/[\wp_{1,3}(\omega)-e_{3,1,3}]$, so that by using
the expansion of $\wp_{1,3}$ at $0$, namely, $\wp_{1,3}(\omega)=1/\omega^{2}+g_{2,1,3}
\omega^{2}/20+g_{3,1,3}\omega^{4}/28+O(\omega^{6})$, $g_{2,1,3}$ and $g_{3,1,3}$ being
the invariants of $\wp_{1,3}$, as well as the straightforward equality $(e_{3,1,3}-e_{1,1,3})(e_{3,1,3}-e_{1+3,1,3})=3{e_{3,1,3}}^{2}-g_{2,1,3}/4$, we get
     \begin{equation}
     \label{first_expansion_e_e_e}
          \wp_{1,4}(\omega)=1/\omega^{2}+[3{e_{3,1,3}}^{2}-g_{2,1,3}/5]\omega^{2}
          +[g_{3,1,3}+3{e_{3,1,3}}^{3}-g_{2,1,3}e_{3,1,3}/4]\omega^{4}
          +O(\omega^{6}).
     \end{equation}

Second, by applying the identity \eqref{principle_transformation} for $p=3$ we obtain
$\wp_{1,4}(\omega)=-\wp(\omega)+[\wp'(\omega)^{2}+\wp'(\omega_{2}/3)^{2}]/[
2(\wp(\omega)-\wp(\omega_{2}/3))]-4\wp(\omega_{2}/3)$. Using that $\wp(
\omega_{2}/3)=1/3$ as well as \eqref{diff_eq_wp} yields
     \begin{equation}
     \label{second_expansion_e_e_e}
          \wp_{1,4}(\omega)=1/\omega^{2}+[2/3-9g_{2}/20]\omega^{2}
          +[10/27-g_{2}/2-27g_{3}/28]\omega^{4}
          +O(\omega^{6}).
     \end{equation}

By identifying the expansions \eqref{first_expansion_e_e_e} and
\eqref{second_expansion_e_e_e} we obtain the expressions 
of $g_{2,1,3}$ and $g_{3,1,3}$ in terms of $g_{2}$, $g_{3}$ and $e_{3,1,3}$.

In addition, $e_{3,1,3}$ is obviously a solution of $4{e_{3,1,3}}^{3}
-g_{2,1,3}e_{3,1,3}-g_{3,1,3}=0$. If we use the expressions of
$g_{2,1,3}$ and $g_{3,1,3}$ obtained just above, we conclude
that ${e_{3,1,3}}^{3}+[9g_{2}/16-5/6]e_{3,1,3}+[35/108-27g_{3}/32-7g_{2}/16]=0$.
We can solve this equation (we recall that $g_{2}$ and $g_{3}$ are known explicitly
with respect to $z$) and, this way, we get $e_{3,1,3}$. Next we obtain $g_{2,1,3}$
and $g_{3,1,3}$, or, equivalently, $e_{1,1,3}$ and $e_{1+3,1,3}$. In particular, the expansion
\eqref{wp_13_wp_14} is now completely known and Theorem \ref{big_theo} \ref{prop_CGF_order_6_+} is proved.
\end{proof}

\section{Conclusions and perspectives}
\label{A_concluding_remark}

%
%
%
%
%
%
%
%

To conclude, we would like to mention some open problems
and perspectives related to this paper.

\begin{itemize}
\item[---] In the finite group case, finding the way to generate a vanishing differential equation
           for the counting function starting from
           Theorem \ref{explicit_integral} is an open problem; on the other hand,
           it is certainly possible to verify \emph{a posteriori} that the function
           satisfies a certain differential equation.
\item[---] In the infinite group case, it is an open problem to find the asymptotic, as $n\to \infty$, of $q_{i,j,n}$;
           see in particular the conjectures of Bostan and Kauers \cite{BKA} on the total numbers of walks  $\sum_{i,j\geq 0} q_{i,j,n}$.
           The singularity analysis of the expression of $Q(x,y;z)$ obtained in Theorem \ref{explicit_integral}---generically delicate, see \cite[Chapter F]{mythesis}---should lead to these results.
\item[---] Last but not least, proving from
           Theorem \ref{explicit_integral} the conjecture
           of Bousquet-Mélou and Mishna---namely, that in the infinite group case,
           $Q(x,y;z)$ is non-holonomic---is also an open problem.
           However, everything suggests that this conjecture is true---like the fact that all
           known examples of infinite group \cite{ALR,MM2} have non-holonomic counting functions,
           or else the link, suggested by Figure \ref{Tables},
           between the nature of $Q(x,y;z)$ and that of $w(x;z)$ and $\widetilde{w}(y;z)$.
           A way to prove this conjecture consists in studying precisely,
           via Theorem \ref{explicit_integral}, the singularities
           of $Q(x,y;z)$.
\end{itemize}

 \begin{figure}[!ht]
\begin{center}
\begin{tabular}{|c|c|c|c|c|}
\hline
   Group  & Covariance & Walks

  & {Nature of $Q$} & Nature of $w$ and $\widetilde{w}$ \\
  \hline
  \hline
  & & \includegraphics[width=11mm,height=11mm]{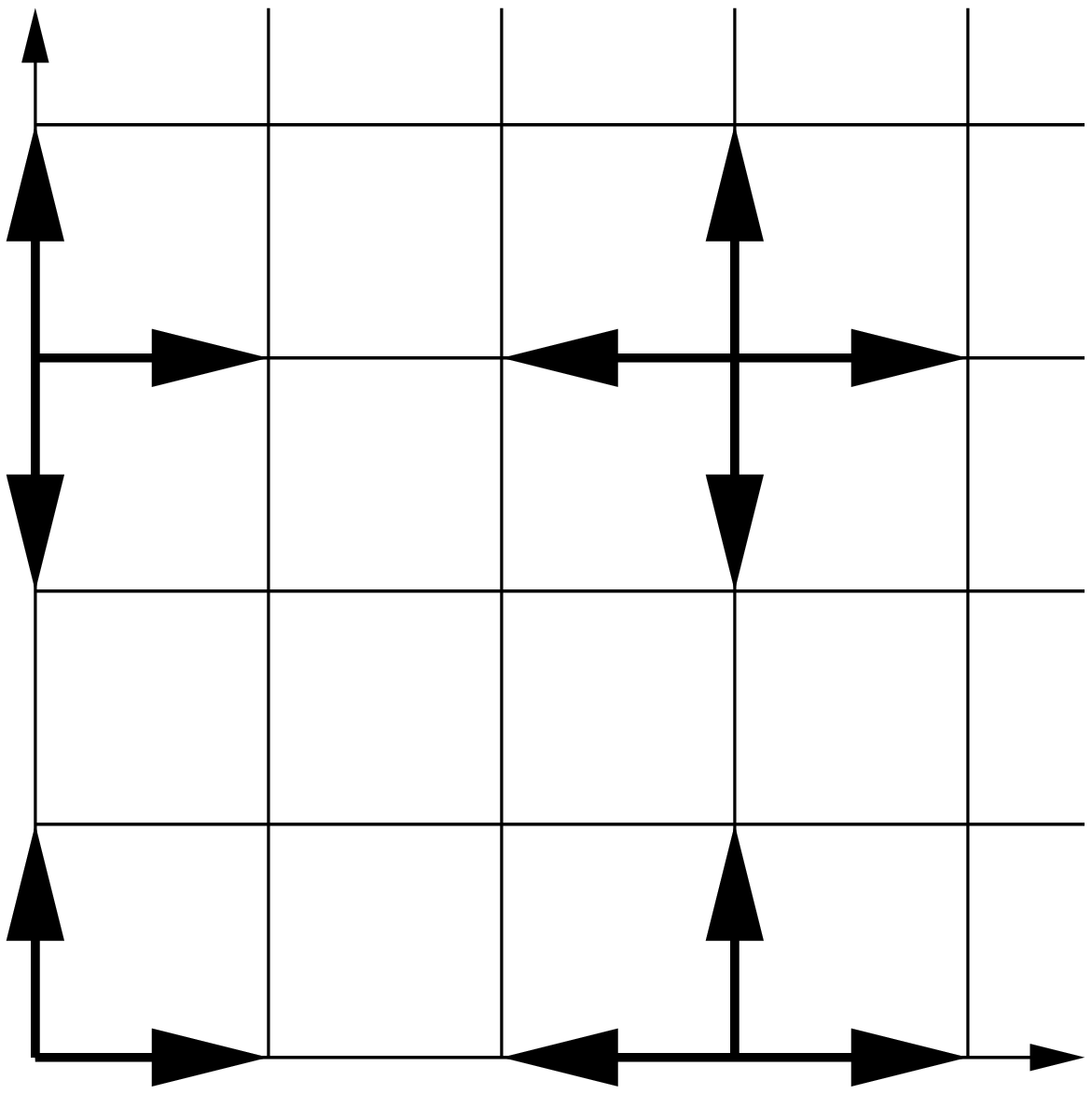}
  \includegraphics[width=11mm,height=11mm]{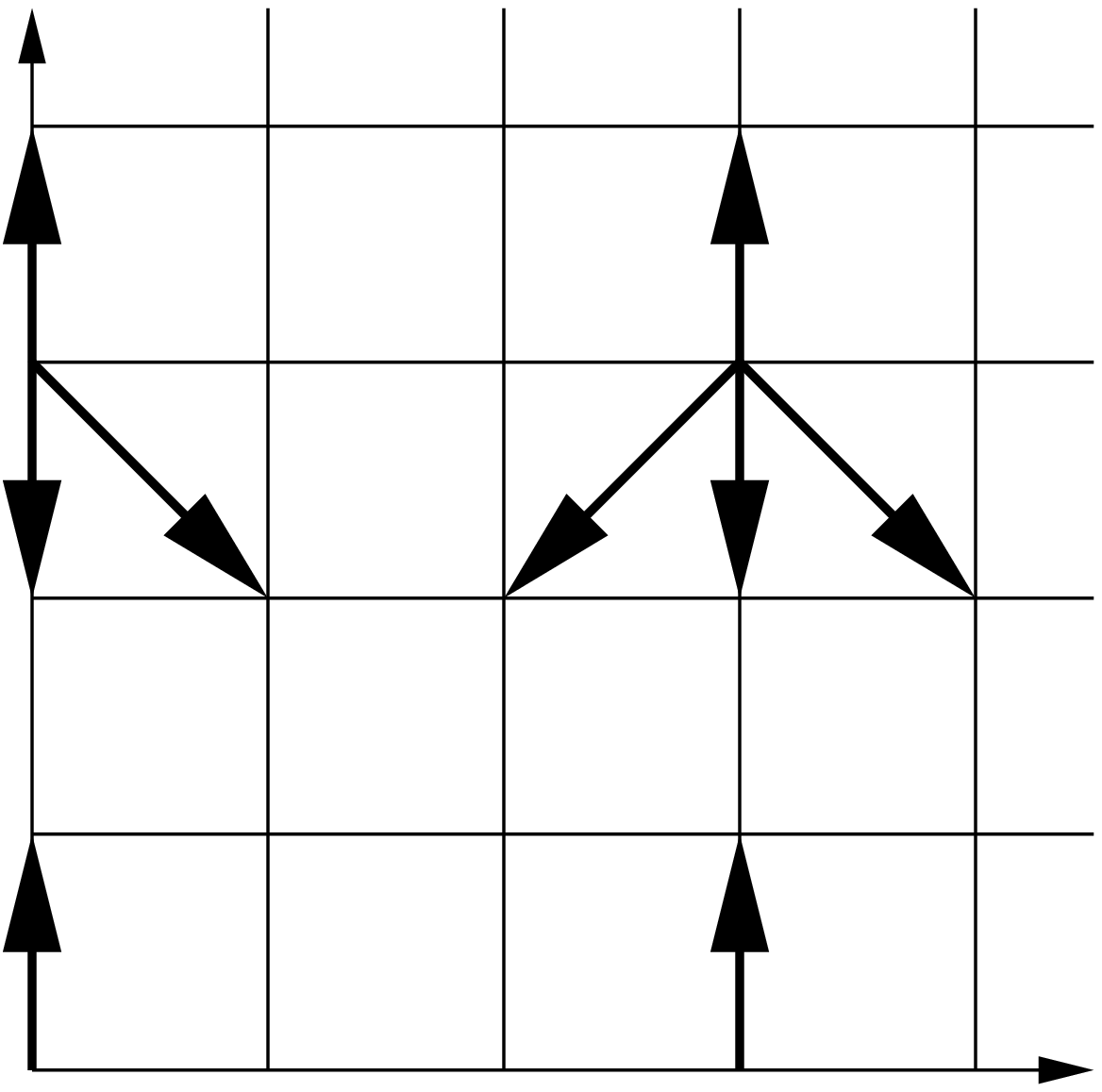}
    & holonomic & rational \\ 
   $4$ & $=0$  & and $14$ others \cite{BMM} 
    &  \cite{BMM,FR} & [Theorems \ref{nature_CGF} and \ref{big_theo} \ref{prop_CGF_order_4}] \\ 
   \hline
   
 $6$ & $<0$ &  \includegraphics[width=11mm,height=11mm]{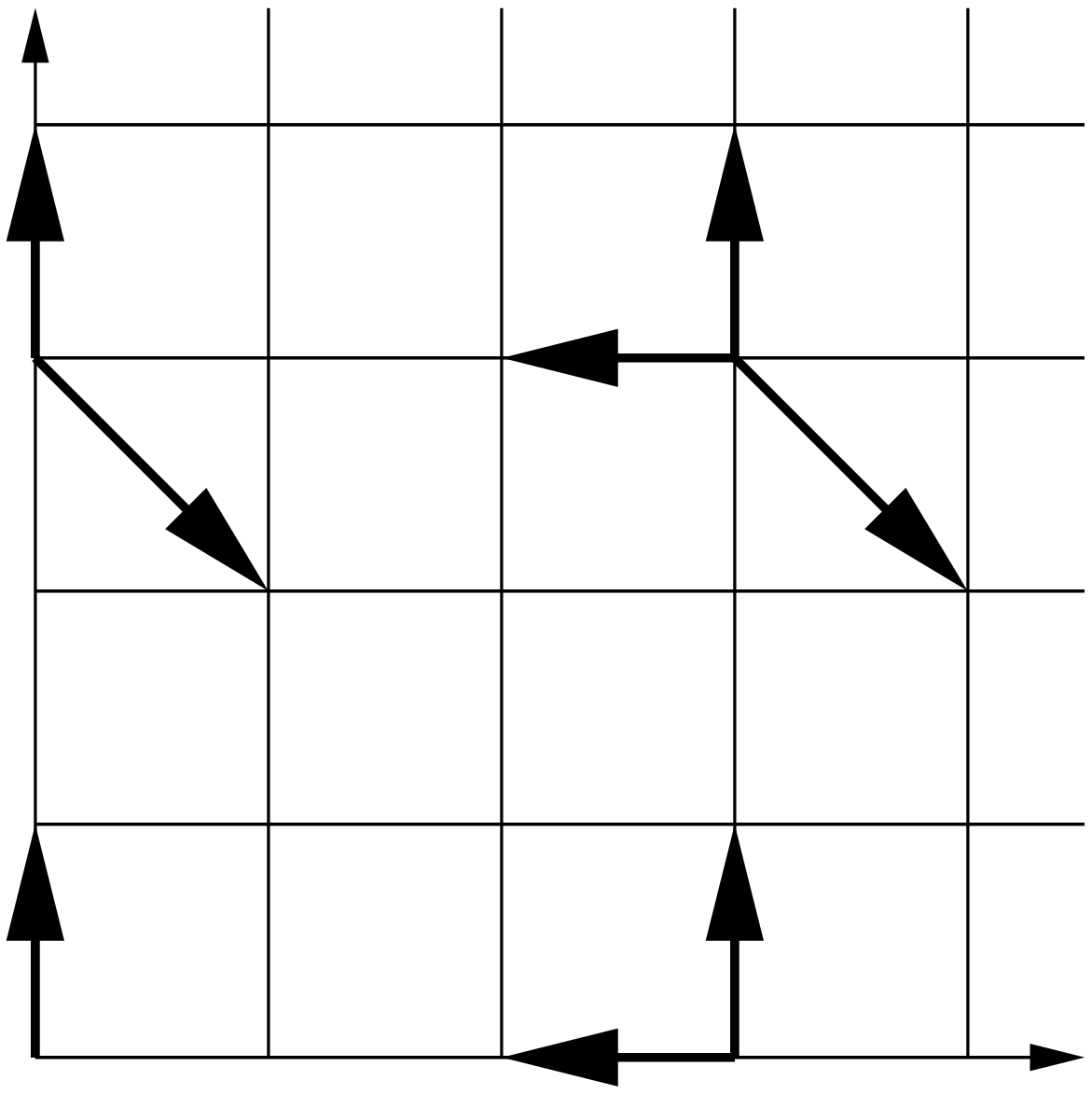}
    \includegraphics[width=11mm,height=11mm]{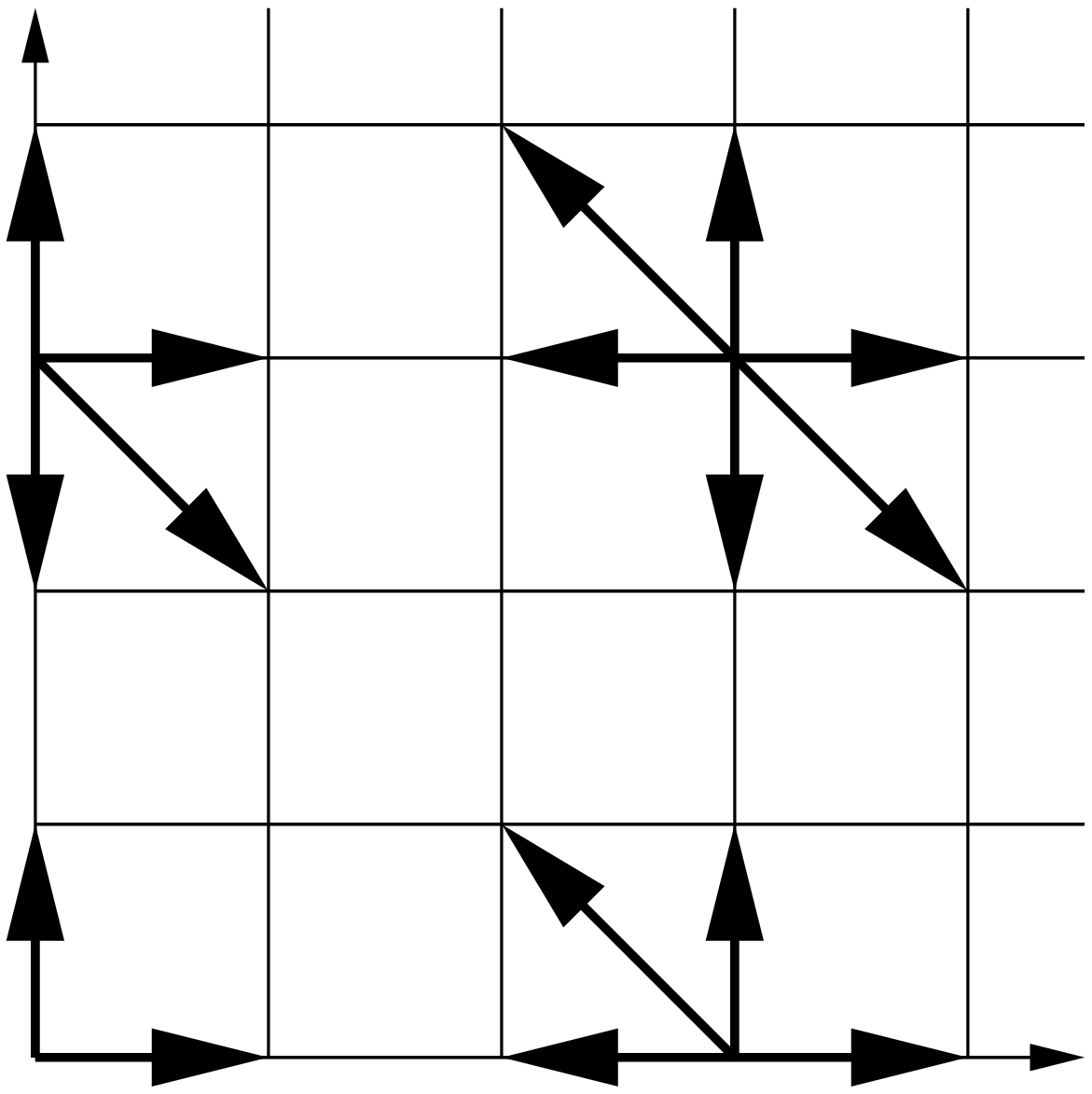}    & 
    \begin{minipage}[b]{1.8cm}
\begin{center}holonomic

\cite{BMM,FR} 
\end{center}
\end{minipage} &     \begin{minipage}[b]{4cm}
\begin{center}rational

[Theorems \ref{nature_CGF} and \ref{big_theo} \ref{prop_CGF_order_6_-}]
\end{center}
\end{minipage}
 \\
   \hline  
 $8$ & $<0$ &  \includegraphics[width=11mm,height=11mm]{gi_8_2.eps}    &     \begin{minipage}[b]{1.8cm}
\begin{center}holonomic

\cite{BMM,FR} 
\end{center}
\end{minipage}&    \begin{minipage}[b]{4cm}
\begin{center}rational

[Theorems \ref{nature_CGF} and \ref{big_theo} \ref{prop_CGF_order_8}]\end{center}
\end{minipage} \\
   \hline\hline
    
 $6$ & $>0$&   \includegraphics[width=11mm,height=11mm]{gi_6_4.eps}
    \includegraphics[width=11mm,height=11mm]{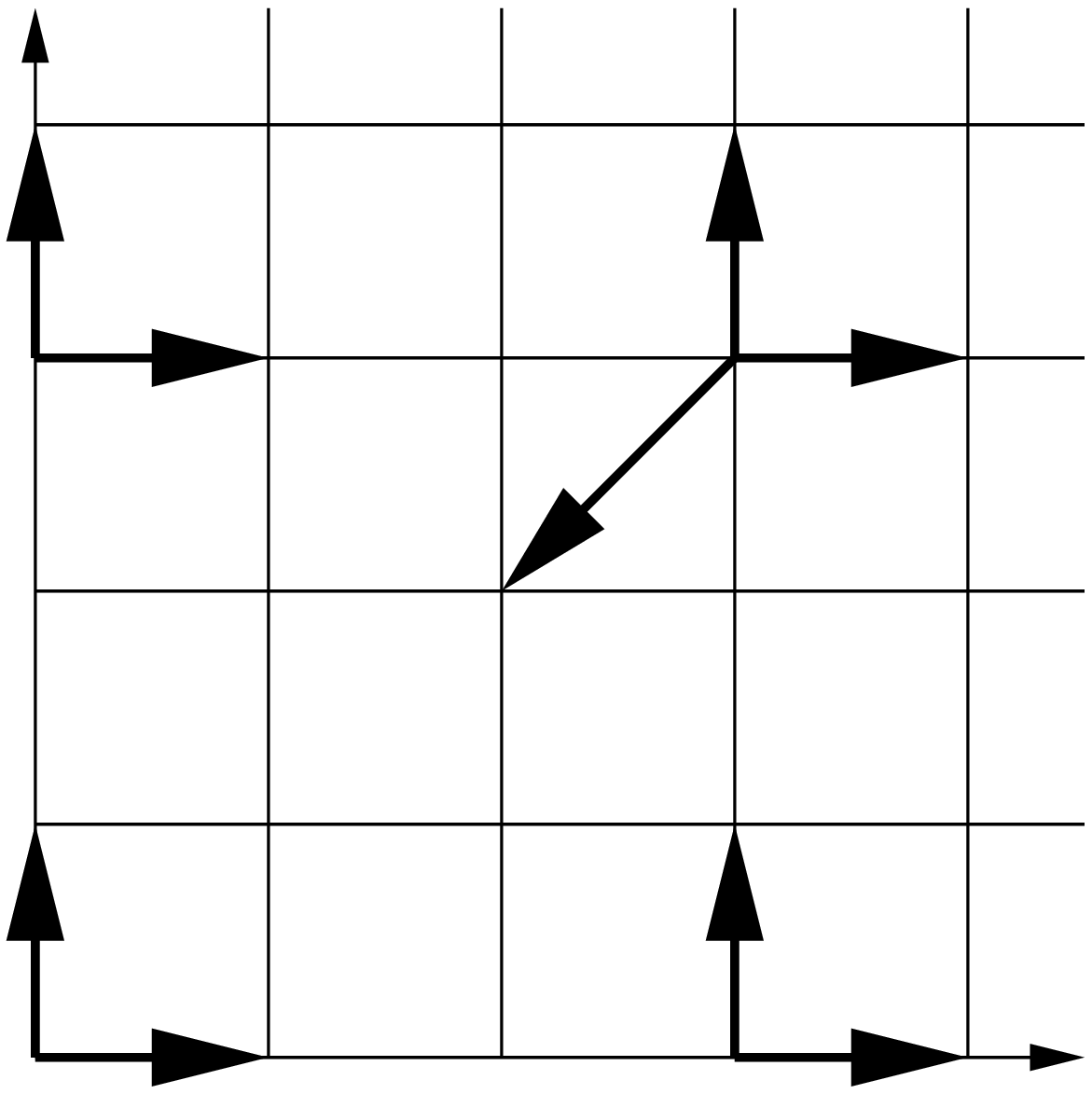}
    \includegraphics[width=11mm,height=11mm]{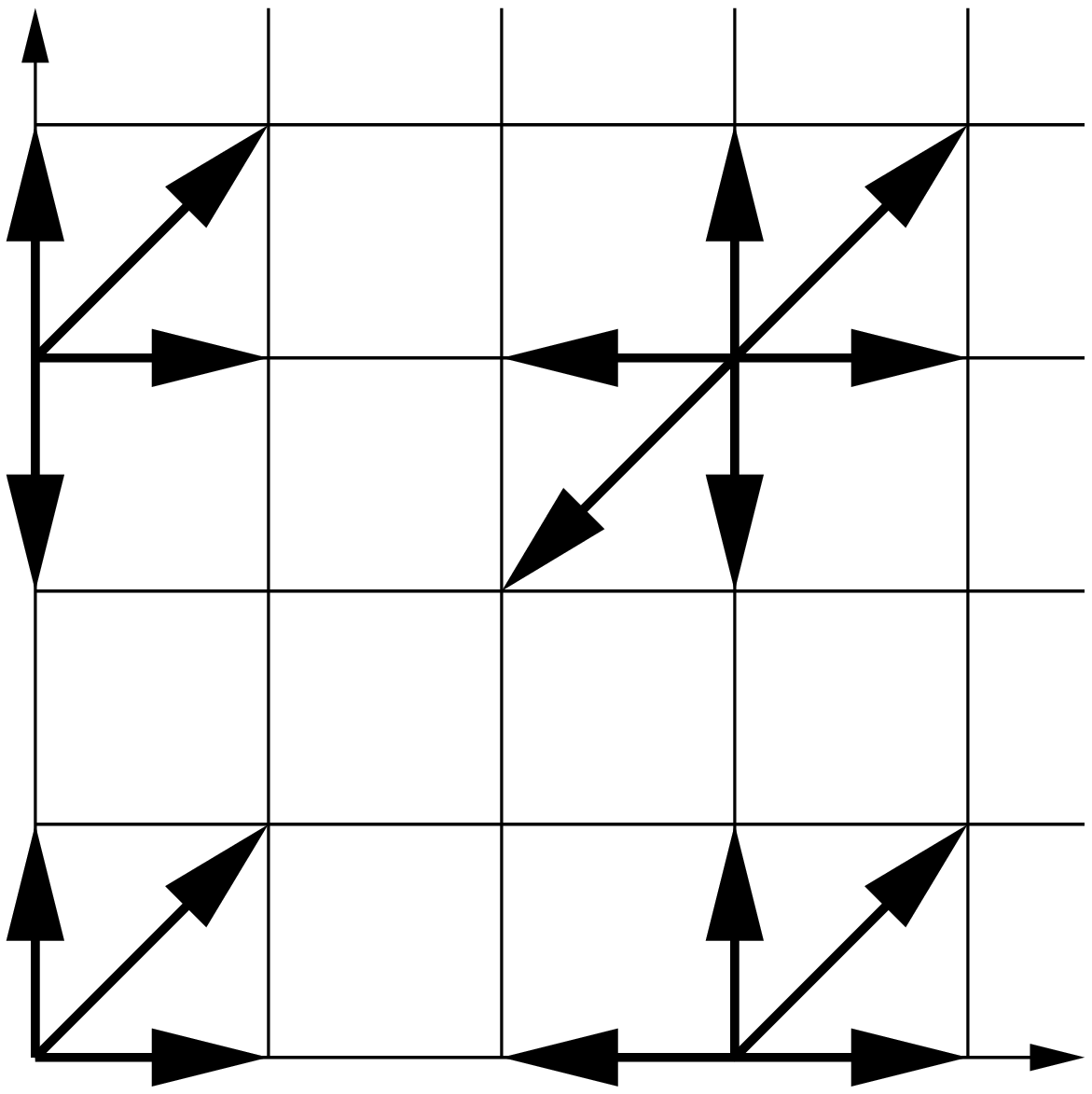}   & 
    \begin{minipage}[b]{1.8cm}
    \begin{center}algebraic
\cite{BMM,FR} 
\end{center}
\end{minipage} &     \begin{minipage}[b]{4cm}
\begin{center}algebraic

[Theorems \ref{nature_CGF} and \ref{big_theo} \ref{prop_CGF_order_6_+}]
\end{center}
\end{minipage}\\
      \hline
      
 $8$  &$>0$ &    \includegraphics[width=11mm,height=11mm]{gi_8_1.eps}    & 
 \begin{minipage}[b]{1.8cm}
 \begin{center}algebraic
\cite{BK,FR} 
\end{center}
\end{minipage} &    \begin{minipage}[b]{1.8cm}
\begin{center}algebraic

\cite{KRG}
\end{center}
\end{minipage} \\
 \hline    \hline  
   & \begin{minipage}[b]{1cm}
\begin{center}

$=0$

$<0$
\end{center}
\end{minipage}&     \includegraphics[width=11mm,height=11mm]{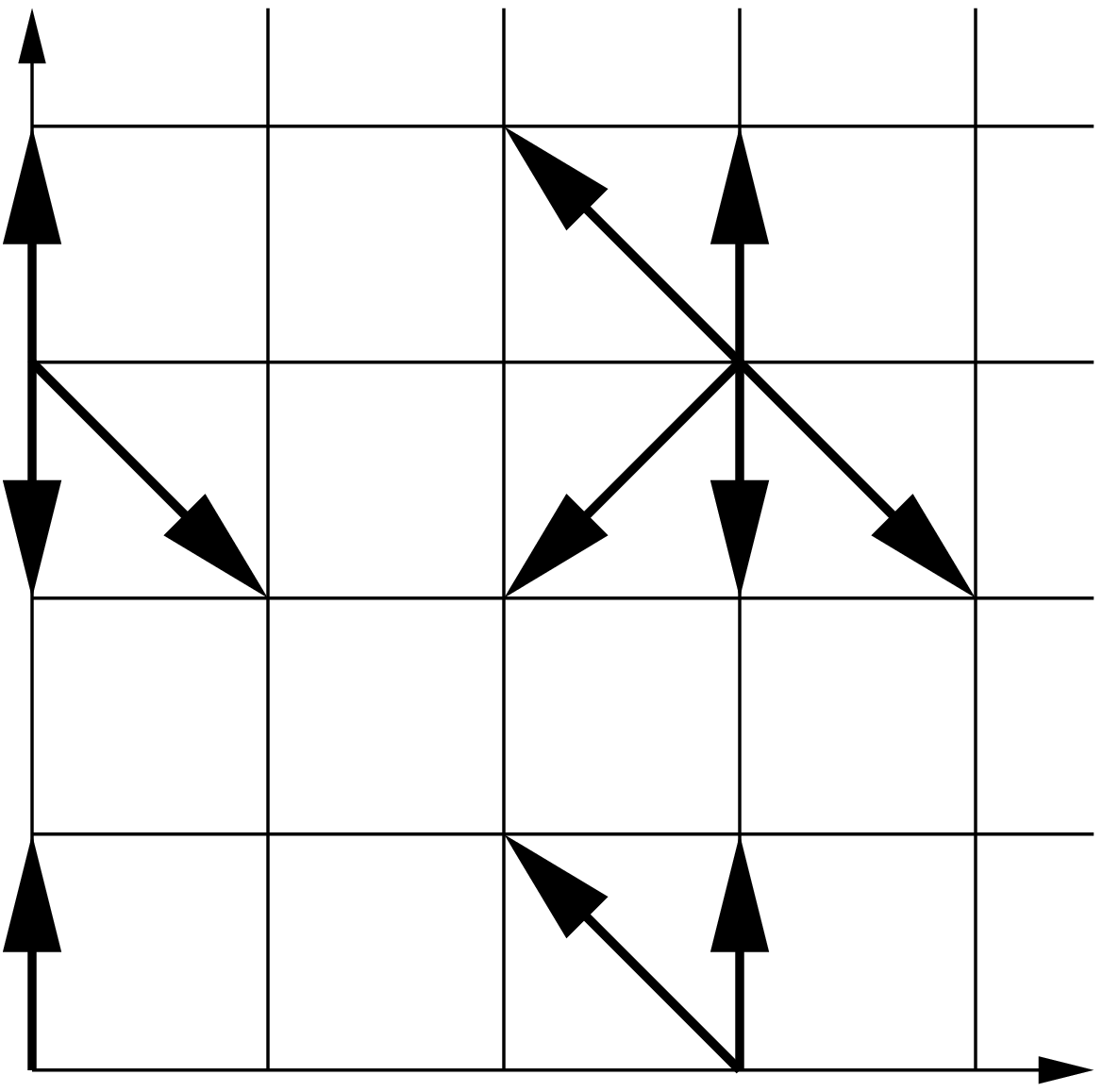}     &  &  non-holonomic \\
    $\infty$ &     $>0$ &   and $50$ others \cite{BMM}   & \textbf{?} & [Theorem \ref{nature_CGF}] \\
 \hline   
\end{tabular}

\bigskip

\caption{Comparison \emph{a posteriori} between the classification of the $74$
non-singular walks according to the nature of the CGFs $w$ and $\widetilde{w}$
and the classification (still incomplete) according to the nature of
the series $Q$, obtained from \cite{BK,BMM,FR}}
\label{Tables}\end{center}
\end{figure}

\medskip

{\small

\noindent{\it Acknowledgments.}
I would like to thank the members of the research team Algorithms of INRIA Rocquencourt,
as well as Guy Fayolle, for the interesting discussions we had together regarding the
topic of this article. I also would like to thank Irina Kurkova for her constant help and her
encouragements all along the elaboration of this work.
I thank Alexis Chommeloux for his
numerous remarks about the English. Finally, 
I must thank two anonymous referees for their so careful reading and their
valuable comments and suggestions.

}

\end{document}